\documentclass[11pt]{amsart}
\usepackage{graphicx}

\usepackage{amsfonts,amsthm,amsmath,amssymb,amscd,mathrsfs}

\usepackage{indentfirst}
\usepackage{cite}
\usepackage{latexsym}

\usepackage{xcolor}
\usepackage{hyperref}
\usepackage{bookmark}
\usepackage{microtype}

\usepackage{lmodern}
\usepackage{enumerate}

\usepackage{tikz}
\usetikzlibrary{arrows.meta}

\usepackage[letterpaper,margin=1in]{geometry}
\numberwithin{equation}{section}
\allowdisplaybreaks[1]
\newtheorem{theorem}{Theorem}[section]
\newtheorem{corollary}{Corollary}[section]
\theoremstyle{remark}
\newtheorem{remark}{Remark}[section]
\theoremstyle{definition}
\newtheorem{definition}{Definition}[section]
\theoremstyle{plain}
\newtheorem{lemma}[theorem]{Lemma}
\newtheorem{proposition}[theorem]{Proposition}
\theoremstyle{definition}
\newtheorem{example}{Example}[section]
\theoremstyle{plain}

\newcommand{\bt}{\begin{theorem}}
	\newcommand{\bl}{\begin{lemma}}
		\newcommand{\el}{\end{lemma}}
	\newcommand{\et}{\end{theorem}}

\newcommand{\curl}{\mathrm{curl}}
\newcommand{\te}{\theta}

\newcommand{\bn}{\begin{eqnarray}}
	\newcommand{\en}{\end{eqnarray}}
\newcommand{\bnn}{\begin{eqnarray*}}
	\newcommand{\enn}{\end{eqnarray*}}

\newcommand{\ba}{\begin{aligned}}
	\newcommand{\ea}{\end{aligned}}
\newcommand{\be}{\begin{equation}}
	\newcommand{\ee}{\end{equation}}

\newcommand{\p}{\partial}

\newcommand{\vp}{\varphi}

\newcommand{\Bv}{{\boldsymbol{v}}}

\newcommand{\Bu}{{\boldsymbol{u}}}

\newcommand{\Be}{{\boldsymbol{e}}}
\newcommand{\BF}{{\boldsymbol{F}}}

\newcommand{\Bx}{{\boldsymbol{x}}}

\newcommand{\R}{\mathbb{R}}

\newcommand{\mcC}{\mathcal{C}}

\newcommand{\rmd}{\mathrm{d}}

\title[Self-similar flows with large swirl forcing]{Self-similar stationary Navier--Stokes flows in a half-space with arbitrarily large swirl forcing}
\date{}

\author{Yun Wang}
\address{School of Mathematical Sciences, Center for dynamical systems and differential equations, Soochow University, Suzhou, China}
\email{ywang3@suda.edu.cn}

\author{Chunjing Xie}
\address{School of Mathematical Sciences,  Ministry of Education Key Laboratory of Scientific and Engineering Computing, and CMA-Shanghai, Shanghai Jiao Tong University, 800 Dongchuan Road, Shanghai, China}
\email{cjxie@sjtu.edu.cn}

\author{Shaoheng Zhang}
\address{School of Mathematical Sciences, Soochow University, Suzhou, China}
\email{20234007008@stu.suda.edu.cn}

\begin{document}

	\begin{abstract}
		We study axisymmetric self-similar solutions of the stationary Navier--Stokes equations with $(-3)$-homogeneous forces and no-slip boundary conditions in a half-space.
		For a fixed force without swirl, we characterize the interval of amplitudes for which solutions exist and describe the velocity singularity at its finite endpoints.
		For each parameter in this interval, a solution exists after adding a swirl force of any size, and it is unique among solutions with bounded profiles when the swirl force is sufficiently small.
		The proof uses the nonnegative contribution of swirl to a Riccati equation and a comparison estimate.
		We also prove that, when the tangential curl of the force is
		sufficiently small, there exists a unique solution with small velocity profiles.
	\end{abstract}

	\maketitle
	
	\section{Introduction}
	
	\subsection{Problem and scaling}
	
	We study the stationary incompressible Navier--Stokes equations
	in the three-dimensional upper half-space
	\[
	\mathbb R^3_+
	:=
	\{\Bx=(x_1,x_2,x_3)\in\mathbb R^3:x_3>0\},
	\]
	subject to the no-slip boundary condition:
	\begin{equation}\label{eqn:NSEhfD}
		\left\{
		\begin{aligned}
			-\Delta \Bu+(\Bu\cdot\nabla)\Bu+\nabla p
			&=\BF
			&&\text{in }\mathbb R^3_+,\\
			\operatorname{div}\Bu
			&=0
			&&\text{in }\mathbb R^3_+,\\
			\Bu
			&=0
			&&\text{on }\partial\mathbb R^3_+\setminus\{0\}.
		\end{aligned}
		\right.
	\end{equation}
	Here $\Bu$, $p$, and $\BF$ denote the velocity, pressure, and
	external force, respectively. The origin is excluded from the
	boundary condition because the homogeneous solutions considered
	below may have a singularity there.
	
	The system is invariant under the scaling
	\[
	\Bu_\kappa(\Bx)=\kappa\Bu(\kappa \Bx),
	\qquad
	p_\kappa(\Bx)=\kappa^2p(\kappa \Bx),
	\qquad
	\BF_\kappa(\Bx)=\kappa^3\BF(\kappa \Bx).
	\]
	A solution is called self-similar if
	$(\Bu_\kappa,p_\kappa)=(\Bu,p)$ for every $\kappa>0$.
	Equivalently,  the velocity and pressure are homogeneous of degrees $-1$
	and $-2$, respectively, 
	\[
	\Bu(\kappa \Bx)=\kappa^{-1}\Bu(\Bx),
	\qquad
	p(\kappa \Bx)=\kappa^{-2}p(\Bx).
	\]
	Consequently, a scale-invariant external force must satisfy
	\[
	\BF(\kappa \Bx)=\kappa^{-3}\BF(\Bx).
	\]
	
	Homogeneous stationary flows describe leading terms in singularity
	and far-field analysis.  Results in three-dimensional exterior domains
	include~\cite{DG00,KS11,DI17,JS26}, while point singularities are studied
	in~\cite{MT12}. Our concern here is the
	existence of force-driven homogeneous flows in a domain with
	a no-slip boundary.
	
	\subsection{Related results}
	
	In $\mathbb R^3\setminus\{0\}$, the classical Landau solutions are
	axisymmetric, swirl-free, and $(-1)$-homogeneous~\cite{Landau44}.
	Up to rotations they form a one-parameter family; see also
	\cite{TianXin98,CK2004,Tsai18}.  Without assuming axisymmetry,
	\v{S}ver\'ak~\cite{Sverak11} proved that every smooth nonzero
	$(-1)$-homogeneous solution is a Landau solution.  Axisymmetric
	homogeneous solutions with singularities at one or both poles of the
	sphere were studied in~\cite{LLY18I,LLY18II,LLY19III, LLY26}.
	
	For the forced whole-space problem, Shi~\cite{Shi18} constructed
	families of axisymmetric homogeneous solutions for sufficiently small
	axisymmetric forces, and obtained
	nonexistence results for large
	multiples of certain non-potential forces.  For a homogeneous solution
	in the whole space, the force must have zero spherical mean; see Remark \ref{compatibility_condition} below.  
	Higher-dimensional rigidity and forced existence results are given in~\cite{BGLWX25, BGLWX23, GLWX26, Ma26}.
	
	Boundaries substantially change the problem. In the half-space,
	Kang, Miura, and Tsai~\cite[Theorem~5.1]{KMT18} proved that an
	axisymmetric $(-1)$-homogeneous unforced solution with the no-slip or
	Navier boundary condition must vanish.  Thus, without an external force, there is no nonzero axisymmetric
	homogeneous solution satisfying the no-slip boundary condition.

	\subsection{Main results}
	In spherical coordinates, we write the velocity and pressure as
	\begin{equation}\label{eqn:Bu-p-explicitformula}
		\Bu
		=
		\frac1\rho
		\bigl(
		f(\varphi)\Be_\rho
		+
		g(\varphi)\Be_\varphi
		+
		h(\varphi)\Be_\theta
		\bigr),
		\qquad
		p=\frac{P(\varphi)}{\rho^2}.
	\end{equation}
	Here $\rho=|\Bx|$, $\varphi$ and $\theta$ are the polar and
	azimuthal angles, and the spherical basis vectors are defined in
	Section~\ref{sec2.1}. The $\Be_\theta$-component is called the swirl
	component.
	
	We fix the angular profiles $X,Y,Z$ and consider forces of the form
	\begin{equation}\label{eqn:two-parameter-force}
		\BF_{\lambda,\mu}
		=
		\frac{\lambda}{\rho^3}
		\bigl(
		X(\varphi)\Be_\rho
		+
		Y(\varphi)\Be_\varphi
		\bigr)
		+
		\frac{\mu}{\rho^3}Z(\varphi)\Be_\theta,
		\qquad
		(\lambda,\mu)\in\mathbb R^2.
	\end{equation}
	By a bounded-profile solution, we mean a solution in the sense of
	Definition~\ref{def:bounded-profile}. In particular, the velocity
	profiles are bounded and the equations hold across the symmetry
	axis.
	
	\begin{theorem}
		\label{thm:main}
		Assume that
		$\rho^{-3}(X(\varphi)\Be_\rho+Y(\varphi)\Be_\varphi)$ and
		$\rho^{-3}Z(\varphi)\Be_\theta$ belong to
		$C^1_{\mathrm{loc}}
		(\overline{\mathbb R^3_+}\setminus\{0\})$.
		Consider \eqref{eqn:NSEhfD} with force
		\eqref{eqn:two-parameter-force}. All the solutions asserted below
		satisfy
		\[
		\Bu\in C^{2,\alpha}_{\mathrm{loc}}
		(\overline{\mathbb R^3_+}\setminus\{0\}),
		\qquad
		p\in C^{1,\alpha}_{\mathrm{loc}}
		(\overline{\mathbb R^3_+}\setminus\{0\})
		\]
		for every $0<\alpha<1$.
		
		\noindent
		(i) \emph{(The no-swirl case.)}
		When $\mu=0$, let $I$ be the set of $\lambda\in\mathbb R$ for which
		\eqref{eqn:NSEhfD} with force $\BF_{\lambda,0}$ admits a
		bounded-profile axisymmetric self-similar solution. Then
		\[
		I=(\underline\lambda,\overline\lambda),
		\qquad
		-\infty\leq\underline\lambda<0
		<
		\overline\lambda\leq+\infty.
		\]
		For every $\lambda\in I$, the solution is unique and has no swirl,
		whereas no bounded-profile solution exists for $\lambda\notin I$.
		To describe the endpoints of $I$, set
		\begin{equation}\label{eqn:disfcn}
			\mathcal N(\varphi)
			=
			\int_0^\varphi
			\bigl(G(\tau)+2\mathcal C\bigr)
			\sin\tau
			\bigl(\cos\varphi-\cos\tau\bigr)
			\,\rmd\tau,
		\end{equation}
		where $G(\varphi) = X(\varphi) + 2\int_0^\varphi Y(\tau)\,\rmd\tau$,
		and 
		$\mathcal C = -\int_0^{\pi/2} G(\tau)\sin\tau\cos\tau\,\rmd\tau$.
		Then
		\begin{itemize}
			\item if $\mathcal{N}\equiv0$, then $\underline{\lambda}=-\infty$ and $\overline{\lambda}=+\infty$;
			\item if $\mathcal{N}\ge0$ and $\mathcal{N}\not\equiv0$, then $\underline{\lambda}\in(-\infty,0)$ and $\overline{\lambda}=+\infty$;
			\item if $\mathcal{N}\le0$ and $\mathcal{N}\not\equiv0$, then $\underline{\lambda}=-\infty$ and $\overline{\lambda}\in(0,+\infty)$;
			\item if $\mathcal{N}$ changes sign, then $\underline{\lambda}\in(-\infty,0)$ and $\overline{\lambda}\in(0,+\infty)$.
		\end{itemize}
		
		\noindent
		(ii) \emph{(The swirling case.)}
		For every $\lambda\in I$ and every $\mu\in\mathbb R$,
		\eqref{eqn:NSEhfD} with force $\BF_{\lambda,\mu}$ admits at least one
		bounded-profile axisymmetric self-similar solution. For each fixed
		$\lambda\in I$, this solution is unique in the whole bounded-profile
		class whenever
		\[
		|\mu|\,\|Z\|_{C^0[0,\pi/2]}
		<
		\varepsilon_\lambda
		\]
		for some $\varepsilon_\lambda>0$. For every compact interval
		$K\Subset I$, $\varepsilon_\lambda$ can be chosen independently of
		$\lambda\in K$.
	\end{theorem}
	
	\begin{remark}
		Theorem~\ref{thm:main}(ii) only requires the no-swirl problem to be solvable; the no-swirl force need not be small. Solutions may also exist for some $\lambda\notin I$ when a suitable nonzero swirl force is added; see Example~\ref{ex:outside-no-swirl-interval}.
	\end{remark}

	\subsection{Consequences and examples}
	
	We first discuss the no-swirl classification in
	Theorem~\ref{thm:main}(i).
	
	If $\mathcal{N}\equiv0$, then \eqref{eqn:disfcn} forces $G+2\mathcal{C}\equiv0$; hence $G$ is constant, and \eqref{eqn:sphericalcurl} implies
	$\nabla\times\BF_{\lambda,0} = 0$. Therefore the force is a gradient and can be absorbed into the pressure. Then only the trivial solution exists (see \cite[Theorem 5.1]{KMT18}), corresponding to $\underline{\lambda}=-\infty$ and $\overline{\lambda}=+\infty$ in Theorem~\ref{thm:main}(i). For the nontrivial cases where $\mathcal{N}\not\equiv0$, the sign of $\mathcal{N}$ determines the three possible behaviors of the discriminant, which are illustrated by the following examples.

	\begin{example}\label{example:N}
		With $G$ and $\mathcal C$ defined as in
		Theorem~\ref{thm:main}(i), the choices
		\begin{itemize}
			\item $G(\varphi)=\mp\cos(2\varphi)$ give $\mcC=0$ and $\mathcal{N}(\varphi)=\pm\frac{1}{6}\cos\varphi(\cos\varphi+2)(1-\cos\varphi)^2$, which are strictly positive and negative, respectively, on
			$(0,\pi/2)$;
			\item $G(\varphi)=\cos(4\varphi)$ gives $\mcC=\frac16$ and $\mathcal{N}(\varphi)=\frac{2}{15}\cos\varphi(1-\cos\varphi)^2\bigl(-2\cos^3\varphi-4\cos^2\varphi-\cos\varphi+2\bigr)$, which changes sign.
		\end{itemize}
		These examples illustrate the three nontrivial cases in
		Theorem~\ref{thm:main}(i).
	\end{example}

	\begin{remark}[Behavior at finite endpoints]
		At every finite endpoint of $I$, the corresponding no-swirl solution
		is defined for $0<\varphi\leq\pi/2$, and its polar profile satisfies
		\[
		g(\varphi)
		=
		-\frac{4}{\varphi}+O(\varphi)
		\longrightarrow-\infty
		\qquad\text{as }\varphi\to0^+.
		\]
		Thus the velocity profile becomes unbounded at the symmetry axis;
		see Lemmas~\ref{lm:exactblow} and
		\ref{lm:endpoint-asymptotics}.
	\end{remark}
	
	We next give two consequences of the swirling result in
	Theorem~\ref{thm:main}(ii).
	
	\begin{corollary}[Pure swirl forcing]\label{cor:pure-swirl}
		For every $\mu\in\mathbb R$ and every axisymmetric force
		$\mu\rho^{-3}Z(\varphi)\Be_\theta$ satisfying the regularity
		assumption of Theorem~\ref{thm:main}, there exists a bounded-profile
		self-similar solution. If $\mu Z\not\equiv0$, then every such solution
		has nontrivial swirl and polar components. More precisely,
		\[
		g(\varphi)>0
		\qquad\text{for }0<\varphi<\frac{\pi}{2}.
		\]
	\end{corollary}
	
	\begin{corollary}[Small tangential curl]
		\label{cor:small-tangential-curl}
		Let $\BF\in C^1_{\mathrm{loc}}(\overline{\mathbb R^3_+}\setminus\{0\})$
		be axisymmetric and $(-3)$-homogeneous, and let
		$(\curl\BF)^{\tan}$ denote the tangential component of $\curl\BF$
		on the unit sphere. There exist absolute constants
		$\varepsilon,C>0$ such that, if
		\[
		\eta
		:=
		\|(\curl\BF)^{\tan}\|_
		{C^0(\overline{\mathbb S^2\cap\mathbb R^3_+})}
		\leq\varepsilon,
		\]
		then \eqref{eqn:NSEhfD} has a \emph{unique} bounded-profile axisymmetric
		self-similar solution. Its velocity profiles satisfy
		\begin{equation}\label{eqn:small-curl-velocity-bound}
			\|f\|_{C^0[0,\pi/2]}
			+
			\|g\|_{C^1[0,\pi/2]}
			+
			\|h\|_{C^1[0,\pi/2]}
			\leq C\eta.
		\end{equation}
	\end{corollary}

	We next make two observations about
	Corollary~\ref{cor:small-tangential-curl}: vanishing tangential curl
	forces the velocity to be zero, whereas small tangential curl need
	not imply small full curl.
	
	\begin{remark}[Vanishing tangential curl]
		As a limiting case of Corollary~\ref{cor:small-tangential-curl}, if $\|(\curl\,\BF)^{\tan}\|_{C^0(\mathbb{S}^2\cap\mathbb{R}^3_+)}=0$, then the whole curl vanishes and the only possible self-similar solution is $\Bu\equiv0$.
		Indeed, for any axisymmetric force $\BF=\rho^{-3}(X\Be_\rho+Y\Be_\varphi+Z\Be_\theta)$, \eqref{eqn:sphericalcurl} yields
		\[
		\nabla\times\BF = \frac{1}{\rho^4}\Bigl( \frac{(Z\sin\varphi)'}{\sin\varphi}\,\Be_\rho + 2Z\,\Be_\varphi - (X'+2Y)\,\Be_\theta \Bigr).
		\]
		Hence $\|(\curl\,\BF)^{\tan}\|_{C^0}=0$ forces $Z=0$ and $X'+2Y=0$, so $\curl\,\BF=0$.
		In fact, $\BF=\nabla\widetilde p$ with
		$\widetilde p=-X(\varphi)/(2\rho^2)$, since $Y=-X'/2$,
		so the force can be absorbed into the pressure. The conclusion then
		follows from \cite[Theorem~5.1]{KMT18}.
	\end{remark}
	
	Smallness of the tangential curl does not imply smallness of the full curl: the radial component $\frac{(Z\sin\varphi)'}{\sin\varphi}$ may be large. The following example illustrates this.
	
	\begin{example}
		Let $\chi\in C_c^\infty(\mathbb R)$ satisfy
		$\operatorname{supp}\chi\subset[-1,1]$, $0\le\chi\le1$, and $\chi(0)=1$.
		For $0<\delta<\pi/8$, set
		\[
		Z_\delta(\varphi)=\sqrt{\delta}\,
		\chi\Bigl(\frac{\varphi-\frac{\pi}{4}}{\delta}\Bigr),
		\qquad
		X_\delta=Y_\delta\equiv0,
		\]
		so that $\operatorname{supp}Z_\delta\subset(\frac{\pi}{8},\frac{3\pi}{8})$.
		Then $\BF_\delta=\rho^{-3}Z_\delta\Be_\theta$ is smooth and axisymmetric. Formula \eqref{eqn:sphericalcurl} gives
		\[
		(\curl\BF_\delta)^{\tan}=\frac{2Z_\delta}{\rho^4}\Be_\varphi,
		\qquad
		\curl\BF_\delta=\frac{1}{\rho^4}
		\Bigl(
		\frac{(Z_\delta\sin\varphi)'}{\sin\varphi}\Be_\rho
		+2Z_\delta\Be_\varphi
		\Bigr).
		\]
		Since $\cot\varphi$ is bounded on $\operatorname{supp}Z_\delta$,
		\[
		\left\|\frac{(Z_\delta\sin\varphi)'}{\sin\varphi}\right\|_{C^0}
		\ge \|Z_\delta'\|_{C^0}-C\|Z_\delta\|_{C^0}
		\ge \delta^{-1/2}\|\chi'\|_{C^0}-C\sqrt{\delta}.
		\]
		As $\delta\to0^+$,
		\[
		\|(\curl\BF_\delta)^{\tan}\|_{C^0(\mathbb S^2\cap\mathbb R^3_+)}
		=2\|Z_\delta\|_{C^0}
		=2\sqrt{\delta}\to0,
		\]
		while
		\[
		\|\curl\BF_\delta\|_{C^0(\mathbb S^2\cap\mathbb R^3_+)}
		\gtrsim\delta^{-1/2}\to\infty.
		\]
		Thus smallness of the tangential curl does not control the full curl.
	\end{example}

	We conclude with the
	extension to solid cones and a comparison with the whole-space
	problem.
	
	\begin{remark}[Extension to solid cones]\label{rem:solid-cones}
		Theorem~\ref{thm:main} and
		Corollaries~\ref{cor:pure-swirl} and
		\ref{cor:small-tangential-curl} extend to every fixed solid cone
		\[
		\mathcal K_{\varphi_0}
		=
		\left\{
		x\in\mathbb R^3\setminus\{0\}:
		\frac{x_3}{|x|}>\cos\varphi_0
		\right\},
		\qquad
		0<\varphi_0<\pi,
		\]
		with the no-slip condition on $\varphi=\varphi_0$. 
		The corresponding constants may depend on $\varphi_0$.
	\end{remark}

	\begin{remark}[Whole-space compatibility condition]
		\label{compatibility_condition}
		For a smooth homogeneous solution on $\mathbb R^3\setminus\{0\}$,
		the flux of the $(-2)$-homogeneous tensor
		$-\nabla\Bu+\Bu\otimes\Bu+pI$ through a sphere is independent of
		its radius. Integration of the momentum equation over a spherical
		shell therefore forces $\int_{\mathbb S^2}\BF\,\mathrm dS=0$.
		For an axisymmetric force, this is equivalent to
		\[
		\int_0^\pi X(\varphi)\sin\varphi\cos\varphi\,\mathrm d\varphi
		=
		\int_0^\pi Y(\varphi)\sin^2\varphi\,\mathrm d\varphi;
		\]
		see also \cite[eq.~(1.4)]{Shi18}.
		In the half-space, the boundary contributes to the momentum balance,
		so no corresponding compatibility condition is required.
	\end{remark}

	\subsection{Main ideas of the proofs}
	
	The divergence-free condition determines $f$ from $g$. Eliminating
	the pressure and setting
	\[
	t=\cos\varphi,
	\qquad
	L(t)=g(\varphi)\sin\varphi,
	\]
	reduce the no-swirl problem to a scalar second-order equation, which
	can be normalized as the Riccati problem
	\[
	\frac{\rmd M}{\rmd t}
	+
	\frac12M^2
	=
	\lambda N_0,
	\qquad
	M(0)=0.
	\]
	Here $N_0$ depends only on $X$ and $Y$. By
	Lemma~\ref{lm:reconstruction}, a bounded-profile solution exists
	exactly when $M$ extends to $t=1$. Comparison and continuous
	dependence show that the admissible values of $\lambda$ form an open
	interval, while the sign of $\mathcal N$ determines its finite
	endpoints.
	
	In the presence of swirl, the Riccati forcing has the form
	\[
	N=\lambda N_0+Q[h],
	\qquad
	Q[h]\geq0.
	\]
	For $\lambda\in I$, let $M_*$ be the no-swirl solution.
	Riccati comparison gives $M\geq M_*$. This lower bound allows
	us to construct a barrier for the swirl equation and obtain
	\[
	\|h\|_{C^0}
	\leq
	C_\lambda|\mu|\|Z\|_{C^0},
	\qquad
	0\leq M-M_*
	\leq
	C_\lambda\mu^2\|Z\|_{C^0}^2.
	\]
	The constant $C_\lambda$ depends only on the fixed no-swirl
	solution. The estimate for $h$ uses only the lower bound
	$M\geq M_*$ and does not require an upper bound for $M$.
	These bounds yield an invariant set for the Schauder fixed-point
	map for every $\mu\in\mathbb R$.
	Difference estimates give uniqueness when
	$|\mu|\|Z\|_{C^0}$ is sufficiently small.
	
	\subsection{Outline}
	
	Section~\ref{sec:Pre} defines the solution class, derives the
	axisymmetric reduction, proves the reconstruction result, and records
	the tools used later. Section~\ref{sec:No-Swirl} proves
	Theorem~\ref{thm:main}(i) by determining the no-swirl interval $I$
	and its finite endpoints. Section~\ref{sec:withswirl} proves
	Theorem~\ref{thm:main}(ii) and its corollaries, and concludes with an
	example of a solution with $\lambda\notin I$.
	\section{Axisymmetric reduction and preliminary tools}\label{sec:Pre}

	\subsection{Axisymmetric curl in spherical coordinates}
	\label{sec2.1}
	
	Recall that the spherical coordinates in $\mathbb R^3$ are denoted by
	$(\rho,\te,\vp)$, where
	\[
	\Bx=(\rho\sin \vp \cos \te,\rho \sin \vp \sin \te, \rho \cos \vp ).
	\]
	The corresponding orthonormal basis vectors are
	\begin{align*} 
		\Be_{\rho}=\frac{\Bx}{\rho}, \quad \Be_{\te}=(-\sin \te, \cos \te,0), \quad \Be_{\vp}=\Be_{\te}\times \Be_{\rho}.
	\end{align*}
	Here $\rho=|\Bx|$, $\te \in (0,2\pi)$ is the azimuthal angle, and $\vp\in (0,\pi)$ is the polar angle.
	In this coordinate system, a vector field $\Bv$ admits the representation
	\begin{equation*}
		\Bv = v^\rho(\rho,\vp,\te) \Be_\rho + v^\vp(\rho,\vp,\te) \Be_\vp + v^\te(\rho,\vp,\te) \Be_\te,
	\end{equation*}
	where $v^\te$ is the {\it swirl} component. We say $\Bv$ is
	{\it axisymmetric} if it is of the form
	\[
	\Bv = v^\rho(\rho,\vp) \Be_\rho + v^\vp(\rho,\vp) \Be_\vp + v^\te(\rho,\vp) \Be_\te.
	\]
	For such an axisymmetric vector field, the curl is given by
	\begin{align}\label{eqn:sphericalcurl}
		\curl\,\Bv = \frac{1}{\rho\sin\vp}\frac{\p(v^{\te}\sin\vp)}{\p\vp}\Be_{\rho}
		- \frac{1}{\rho}\frac{\p(\rho v^{\te})}{\p\rho}\Be_{\vp}
		+ \frac{1}{\rho}\Bigl(\frac{\p(\rho v^{\vp})}{\p\rho} - \frac{\p v^{\rho}}{\p\vp}\Bigr)\Be_{\te},
	\end{align}
	see, for example,~\cite[eq.~(2.47)]{AW95}.
	We shall use the two-parameter force \eqref{eqn:two-parameter-force}.
	For this force, 
	\begin{equation*}
		\curl\,\BF_{\lambda,\mu}
		=\frac{1}{\rho^4}
		\left(
		\mu\frac{(Z\sin\varphi)'}{\sin\varphi}\Be_\rho
		+2\mu Z\Be_\varphi
		-\lambda(X'+2Y)\Be_\theta
		\right).
	\end{equation*}
	Here primes denote differentiation with respect to $\vp$. Thus
	\[
	(\curl\,\BF_{\lambda,\mu})^{\tan}
	=\frac{1}{\rho^4}
	\bigl(2\mu Z\Be_\varphi-\lambda(X'+2Y)\Be_\theta\bigr).
	\]
	
	\subsection{The solution class and regularity at the axis}
	\label{sec:solution-class}
	
	In Sections~\ref{sec:Pre}--\ref{sec:withswirl} we work exclusively in
	the half-space $\mathbb R^3_+=\{\Bx:x_3>0\}$. Thus
	$0\leq\varphi\leq\pi/2$ and, after the change of variables
	$t=\cos\varphi$, $0\leq t\leq1$. Regularity of the force is imposed
	on the \emph{three-dimensional vector field}:
	\begin{equation}\label{eqn:force-regularity}
		\BF\in C^1_{\mathrm{loc}}
		(\overline{\mathbb R^3_+}\setminus\{0\}),
		\qquad \BF(\kappa x)=\kappa^{-3}\BF(x)\quad(\kappa>0).
	\end{equation}
	For an axisymmetric force with angular profiles $X,Y,Z$, this implies
	\begin{equation}\label{eqn:force-pole-conditions}
		X,Y,Z\in C^1[0,\pi/2],\qquad
		X'(0)=0,\qquad Y(0)=Z(0)=0.
	\end{equation}
	In particular, $Y(\varphi),Z(\varphi)=O(\varphi)$ at the axis; see \cite{LW09}.

	\begin{definition}[Bounded-profile solutions]\label{def:bounded-profile}
		A bounded-profile solution in $\mathbb R^3_+$ is an
		axisymmetric pair $(\Bu,p)$, homogeneous of degrees $-1$ and $-2$,
		respectively, with the following properties.  For every $0<a<b<\infty$,
		writing $A_{a,b}=\mathbb R^3_+\cap\{a<|\Bx|<b\}$, we have
		\[
		\Bu\in W^{1,2}(A_{a,b})\cap L^\infty(A_{a,b}),
		\qquad p\in L^2(A_{a,b}).
		\]
		The Navier--Stokes equations hold in distributions in
		$\mathbb R^3_+$, including its symmetry axis, and $\Bu$ has zero trace
		on $\{\Bx \in \R^3 :x_3=0\}\setminus\{0\}$.
	\end{definition}

	The interior and boundary regularity estimates for the Stokes
	equations~\cite{Galdi11}, together with the compatibility conditions
	for axisymmetric vector fields in~\cite[Corollary~1, Lemma~2]{LW09}, give the following
	lemma.
	
	\begin{lemma}[Regularity of bounded-profile solutions]
		\label{lm:profile-regularity}
		Under \eqref{eqn:force-regularity}, every solution in
		Definition~\ref{def:bounded-profile} satisfies
		\begin{equation}\label{eqn:profile-regularity}
			\Bu\in C^{2,\alpha}_{\mathrm{loc}}
			(\overline{\mathbb R^3_+}\setminus\{0\}),
			\qquad
			p\in C^{1,\alpha}_{\mathrm{loc}}
			(\overline{\mathbb R^3_+}\setminus\{0\})
		\end{equation}
		for every $0<\alpha<1$.  Its angular profiles satisfy
		\begin{equation}\label{eqn:halfspace-profile-bc}
			f'(0)=g(0)=h(0)=0,\qquad
			f(\pi/2)=g(\pi/2)=h(\pi/2)=0.
		\end{equation}
	\end{lemma}

	\subsection{ODE reduction and reconstruction}
	\label{sec:integrated-reconstruction}
	
	
	Assume that $\Bu$, $p$, and $\BF$ are axisymmetric and homogeneous
	of degrees $-1$, $-2$, and $-3$, respectively, and write
	\[
	\Bu=\frac1\rho
	\bigl(f(\varphi)\Be_\rho+g(\varphi)\Be_\varphi
	+h(\varphi)\Be_\theta\bigr),
	\quad p=\frac{P(\varphi)}{\rho^2}, \quad
	\BF=\frac1{\rho^3}
	\bigl(X(\varphi)\Be_\rho+Y(\varphi)\Be_\varphi
	+Z(\varphi)\Be_\theta\bigr).
	\]
	Substitution into \eqref{eqn:NSEhfD} gives
	\begin{equation}\label{eqn:halfD1}
		\left\{
		\begin{aligned}
			f''+f'\cot\varphi
			&=gf'-(f^2+g^2+h^2)-2P-X,\\
			f'&=gg'-h^2\cot\varphi+P'-Y,\\
			(h'+h\cot\varphi)'&=g(h'+h\cot\varphi)-Z,\\
			f+g'+g\cot\varphi&=0.
		\end{aligned}
		\right.
	\end{equation}
	The endpoint conditions are those of \eqref{eqn:halfspace-profile-bc}.
	We derive a reduction valid for forces with or without swirl, inspired by the one used to derive the Landau solution; see, e.g., \cite[Theorem~8.1]{Tsai18}.
	Integrating \eqref{eqn:halfD1}$_2$ gives
	\begin{equation}\label{eqn:eqnforP}
		\frac12g^2+P
		=f+\int_0^\varphi\bigl(Y(\tau)+h(\tau)^2\cot\tau\bigr)
		\,\mathrm d\tau+\mcC.
	\end{equation}
	Substituting this into \eqref{eqn:halfD1}$_1$ and using
	$f''+f'\cot\varphi = (f'\sin\varphi)'/\sin\varphi$ gives
	\begin{equation*}
		\frac{(f'\sin\varphi)'}{\sin\varphi}
		=gf'-f^2-2f-(G+2\mcC),
	\end{equation*}
	where the force and swirl contributions are collected in
	\begin{equation}\label{eqn:general-G}
		G(\varphi)
		=X(\varphi)+2\int_0^\varphi Y(\tau)\,\mathrm d\tau
		+h(\varphi)^2+2\int_0^\varphi h(\tau)^2\cot\tau\,\mathrm d\tau.
	\end{equation}
	Since $h(\varphi)=O(\varphi)$ near $0$ by Lemma~\ref{lm:profile-regularity}, the last integral above converges at $0$.
	Using $(g\sin\varphi)'=-f\sin\varphi$ from \eqref{eqn:halfD1}$_4$, we obtain
	\[
	-(f'\sin\varphi)'+(fg\sin\varphi)'+(2g\sin\varphi)'
	=(G+2\mcC)\sin\varphi.
	\]
	Integration from $0$ to $\vp$, with $f'(0)=g(0)=0$, yields
	\begin{equation}\label{eqn:reduct2}
		-f'\sin\varphi+fg\sin\varphi+2g\sin\varphi
		=\int_0^\varphi(G(\tau)+2\mcC)\sin\tau\,\mathrm d\tau.
	\end{equation}
	
	Set $L(t)=g(\varphi)\sin\varphi$ with $t=\cos\varphi$.
	Define
	\begin{equation}\label{eqn:def-J}
		J(t)=\int_0^\varphi
		\bigl(G(\tau)+2\mcC\bigr)\sin\tau\,\mathrm d\tau.
	\end{equation}
	Due to \eqref{eqn:halfD1}$_4$, we have
	\[
	\frac{\mathrm dL}{\mathrm dt}=f,
	\quad 
	f'=-\sin\varphi\frac{\mathrm d^2L}{\mathrm dt^2}.
	\]
	Hence \eqref{eqn:reduct2} becomes
	\begin{equation}\label{eqn:reduct3}
		(1-t^2)\frac{\mathrm d^2L}{\mathrm dt^2}
		+L\frac{\mathrm dL}{\mathrm dt}+2L=J,
		\qquad L(0)=\frac{\mathrm dL}{\mathrm dt}(0)=L(1)=0.
	\end{equation}
	The left-hand side is an exact derivative:
	\[
	(1-t^2)\frac{\mathrm d^2L}{\mathrm dt^2}
	+L\frac{\mathrm dL}{\mathrm dt}+2L
	=\frac{\mathrm d}{\mathrm dt}
	\left((1-t^2)\frac{\mathrm dL}{\mathrm dt}
	+\frac12L^2+2tL\right).
	\]
	Since $\mathrm dL/\mathrm dt=f$ is bounded, both $L(t)$ and
	$(1-t^2)\mathrm dL/\mathrm dt$ tend to
	zero as $t\to1^-$.  Integrating \eqref{eqn:reduct3} from $t$ to $1$ therefore gives
	\begin{equation}\label{eqn:reduct4}
		(1-t^2)\frac{\mathrm dL}{\mathrm dt}
		+\frac12L^2+2tL=-\int_t^1J(s)\,\mathrm ds.
	\end{equation}
	Evaluating at $t=0$ and using $L(0)=\frac{\mathrm{d} L}{\mathrm{d} t}(0)=0$ gives the compatibility condition
	$\int_0^1 J(s)\,\mathrm{d}s = 0$, which explicitly is
	\begin{equation}\label{eqn:constC}
		\begin{aligned}
			\int_0^1 J(s) \, \mathrm{d} s 
			&= \int_0^{\frac{\pi}{2}} \biggl(\int_0^{\varphi} (G(\tau) +2 \mcC) \sin\tau \, \mathrm{d}\tau \biggr) \sin\varphi \, \mathrm{d}\varphi\\
			&= \int_0^{\frac{\pi}{2}} G(\tau) \sin\tau \biggl(\int_{\tau}^{\frac{\pi}{2}}
			\sin\varphi \, \mathrm{d}\varphi \biggr) \, \mathrm{d}\tau + \mcC \int_0^{\frac{\pi}{2}} 2\sin\tau(1-\cos\tau)\,\mathrm{d}\tau\\
			&= \int_0^{\frac{\pi}{2}} G(\tau) \sin\tau \cos\tau \, \mathrm{d}\tau + \mcC = 0.
		\end{aligned}
	\end{equation}
	Thus the integration constant is uniquely determined by
	\begin{equation}\label{eqn:halfspace-C}
		\mcC=-\int_0^{\pi/2}G(\tau)\sin\tau\cos\tau\,\mathrm d\tau.
	\end{equation}
	With
	\begin{equation}\label{eqn:def-M-N}
		M(t)=\frac{L(t)}{1-t^2}=\frac{g(\varphi)}{\sin\varphi},
		\qquad
		N(t)=-\frac{1}{(1-t^2)^2}\int_t^1J(s)\,\mathrm ds,
	\end{equation}
	equation \eqref{eqn:reduct4} reduces to
	\begin{equation}\label{eqn:Riccati}
		\frac{\mathrm{d} M}{\mathrm{d} t}+\frac12M^2=N,\qquad t\in[0,1],\qquad M(0)=0.
	\end{equation}
	We have $N(0)=0$ by the choice of $\mcC$; hence $\frac{\mathrm{d} M}{\mathrm{d} t}(0)=0$
	follows automatically from this initial condition.
	
	We verify that $N(t)$ extends continuously to $[0,1]$.
	From
	\[
	\frac{\mathrm{d} J}{\mathrm{d} t}
	= \frac{\mathrm{d} J}{\mathrm{d}\varphi}\frac{\mathrm{d}\varphi}{\mathrm{d} t}
	= -\bigl(G(\varphi)+2\mcC\bigr),
	\]
	we obtain $\frac{\mathrm{d} J}{\mathrm{d} t}(1)=-(G(0)+2\mcC)$. 
	Since $J(1)=0$, expanding $-\int_t^1 J(s)\,\mathrm{d}s$ around $t=1$ gives
	\[
	-\int_t^1 J(s)\,\mathrm{d}s
	= \frac12 \, \frac{\mathrm{d} J}{\mathrm{d} t}(1)\,(t-1)^2 + o\bigl((t-1)^2\bigr)
	= -\frac{G(0)+2\mcC}{2}(t-1)^2 + o\bigl((t-1)^2\bigr).
	\]
	Dividing by $(1-t^2)^2=(1-t)^2(1+t)^2$ yields
	\[
	\lim_{t\to 1^-} N(t) = -\frac{1}{8}\bigl(G(0)+2\mcC\bigr)
	= -\frac{1}{8}\bigl(X(0)+2\mcC\bigr).
	\]
	Thus $N\in C^0[0,1]$, so the Riccati initial-value problem is
	locally well posed. Whether its solution extends to $t=1$ is studied
	below.
	
	
	The calculation above shows that every bounded-profile solution gives
	a pair $(M,h)$ satisfying the reduced equations.  The converse is the
	reconstruction statement needed below.
	
	\begin{lemma}[Reconstruction from the reduced system]\label{lm:reconstruction}
		Assume \eqref{eqn:force-regularity}.  Let
		\[
		M\in C^1[0,1],\qquad h\in C^1[0,\pi/2]
		\]
		solve \eqref{eqn:Riccati}, with $G,\mcC,J,N$ defined by
		\eqref{eqn:general-G}, \eqref{eqn:halfspace-C}, and
		\eqref{eqn:def-M-N}, and suppose that
		\begin{equation}\label{eqn:reconstruction-swirl}
			H=h'+h\cot\varphi\in C^1[0,\pi/2],\qquad
			H'=gH-Z,\qquad h(0)=h(\pi/2)=0,
		\end{equation}
		where $g(\varphi)=M(\cos\varphi)\sin\varphi$ and $H(0)$ is understood
		by continuity.  Define the remaining profiles by
		\begin{align}\label{eqn:reconstruct-profiles}
			\begin{aligned}
				f(\varphi)
				&=-2tM(t)+(1-t^2)\frac{\rmd M}{\rmd t}(t),\qquad t=\cos\varphi,\\
				P(\varphi)
				&=f(\varphi)-\frac12g(\varphi)^2
				+\int_0^\varphi\bigl(Y(\tau)+h(\tau)^2\cot\tau\bigr)
				\,\mathrm d\tau+\mcC.
			\end{aligned}
		\end{align}
		Then these profiles reconstruct a bounded-profile solution in the
		sense of Definition~\ref{def:bounded-profile}.  In particular, the
		fields satisfy \eqref{eqn:profile-regularity}, and no additional
		distributional force is supported on the symmetry axis.
	\end{lemma}
	
	\begin{proof}
		\underline{Step 1.} \emph{Recovery and boundary conditions.}
		Let $t=\cos\varphi$, define
		$g(\varphi)=M(t)\sin\varphi$, and define $f$ and $P$ by
		\eqref{eqn:reconstruct-profiles}. Reversing the calculations in
		Section~\ref{sec:integrated-reconstruction}, we find that
		$f,g,h,P$ satisfy the full system \eqref{eqn:halfD1} for
		$0<\varphi\leq\pi/2$. Indeed, the definitions of $f$ and $P$ give
		the divergence and polar equations, the equation for $H$ gives the
		swirl equation, and the Riccati equation gives the radial equation.
		
		The choice of $\mcC$ gives $N(0)=0$. Since $M(0)=0$, the Riccati
		equation yields $\frac{\rmd M}{\rmd t}(0)=0$. Therefore,
		\[
		g(\pi/2)=M(0)=0,\qquad
		f(\pi/2)=\frac{\rmd M}{\rmd t}(0)=0,\qquad
		h(\pi/2)=0,
		\]
		which proves the no-slip boundary condition.

		\underline{Step 2.} \emph{Regularity at the axis.}
		Since $M\in C^1[0,1]$ and $h\in C^1[0,\pi/2]$ with $h(0)=0$, we have
		\[
		g(\varphi)=O(\varphi),\qquad h(\varphi)=O(\varphi)
		\quad\text{as }\varphi\to0,
		\]
		while $f,P,g'$, and $h'$ are bounded. Moreover,
		$J(\cos\varphi)=O(\varphi^2)$, and hence
		\eqref{eqn:reduct2} gives
		\[
		f'(\varphi)
		=g(\varphi)\bigl(f(\varphi)+2\bigr)
		-\frac{J(\cos\varphi)}{\sin\varphi}
		=O(\varphi).
		\]
		Using the polar equation and $Y(\varphi)=O(\varphi)$, we also obtain
		$P'(\varphi)=O(\varphi)$.
		
		In cylindrical coordinates, write
		\[
		\Bu
		=
		\rho^{-1}
		\bigl(
		u^r(\varphi)\Be_r
		+h(\varphi)\Be_\theta
		+u^z(\varphi)\Be_z
		\bigr),
		\]
		where
		\[
		u^r=f\sin\varphi+g\cos\varphi,
		\qquad
		u^z=f\cos\varphi-g\sin\varphi.
		\]
		The preceding estimates imply $u^r(\varphi),h(\varphi)=O(\varphi), \, u^z(\varphi)=f(0)+O(\varphi^2).$
		Moreover, $\frac{u^r(\varphi)}{\sin\varphi}, \, \frac{h(\varphi)}{\sin\varphi}$
		are bounded near the axis. The standard cylindrical-to-Cartesian
		formulas therefore show that $\Bu$ and $p$ extend across the axis
		with locally bounded first derivatives on every half-annulus away
		from the origin. In particular, they satisfy the regularity
		requirements in Definition~\ref{def:bounded-profile}.
		
		\underline{Step 3.} \emph{Equations across the axis.}
		Step~2 gives $g,h=O(\varphi)$ and bounded
		$f,f',g',h',P$. Direct differentiation of the spherical
		representation yields
		\[
		|\nabla\Bu|
		\leq \frac{C}{\rho^2}
		\left(
		|f|+|f'|+|g'|+|h'|
		+\frac{|g|+|h|}{\sin\varphi}
		\right)
		\leq \frac{C}{\rho^2},
		\qquad
		|p|=\frac{|P|}{\rho^2}\leq\frac{C}{\rho^2}.
		\]
		Thus $\Bu,\nabla\Bu,p$ are uniformly bounded on each fixed
		half-annulus away from the origin.
		
		Set $r=(x_1^2+x_2^2)^{1/2}$ and fix
		$\psi\in C_c^\infty(\mathbb R^3_+;\mathbb R^3)$.
		Testing the momentum equation on
		$\mathbb R^3_+\cap\{r>\varepsilon\}$, the artificial boundary
		term satisfies
		\[
		\left|
		\int_{\mathbb R^3_+\cap\{r=\varepsilon\}}
		\bigl(\partial_n\Bu-(\Bu\cdot n)\Bu-pn\bigr)
		\cdot\psi\,\mathrm dS
		\right|
		\leq C_\psi\varepsilon\longrightarrow0,
		\]
		where $n$ is the outward unit normal to the punctured domain.
		Indeed, the integrand is uniformly bounded, and the boundary
		portion meeting $\operatorname{supp}\psi$ has area
		$O(\varepsilon)$. The same argument applies to the divergence
		equation. Hence the equations hold distributionally across
		the axis. Together with Steps~1--2, this proves that $(\Bu,p)$
		satisfies Definition~\ref{def:bounded-profile}.
		Finally, Lemma~\ref{lm:profile-regularity} gives
		\eqref{eqn:profile-regularity}.
		
		
	\end{proof}

	\subsection{Comparison and fixed-point tools}
	We first record a comparison lemma for the Riccati equation.
	
	\begin{lemma}[Riccati comparison]
		\label{lm:Riccati-comparison}
		Let $a<b$, $N_*, Q \in C^0[a,b]$, and suppose that $Q \ge 0$.
		Assume that
		\[
		\frac{\rmd M_*}{\rmd t}+\frac12M_*^2=N_*,
		\qquad M_*(a)=0,
		\]
		has a solution on $[a,b]$, and let $M$ be the local solution of
		\[
		\frac{\rmd M}{\rmd t}+\frac12M^2=N_*+Q,
		\qquad M(a)=0.
		\]
		Then $M$ extends to $[a,b]$ and satisfies
		\begin{equation}\label{eqn:Riccati-difference-bound}
			\begin{aligned}
				0\leq M(t)-M_*(t)
				\leq 
				e^{-\int_a^tM_*}
				\int_a^t
				e^{\int_a^sM_* }
				Q(s)\,\rmd s.
			\end{aligned}
		\end{equation}
	\end{lemma}
	
	\begin{proof}
		Since $Q\ge0$, the scalar comparison theorem \cite[Chapter~III, Corollary~4.2]{Hartman02} gives $M\ge M_*$. As we also need a quantitative upper bound, we give the short argument.
		
		Subtracting the two Riccati equations yields
		\[
		\frac{\rmd}{\rmd t}(M-M_*)+M_*(M-M_*)+\frac12(M-M_*)^2=Q.
		\]
		Dropping the nonnegative quadratic term and integrating with the factor $\exp(\int_a^t M_*)$ gives \eqref{eqn:Riccati-difference-bound}. Together with $M-M_*\geq0$, this bound allows $M$ to extend
		to $[a,b]$.
	\end{proof}

	We also use the following standard fixed-point theorem; see, for
	example, \cite[Corollary~11.2]{GT01}.

	\begin{lemma}[Schauder fixed-point theorem]\label{lm:Schauder}
		Let $\mathcal K$ be a nonempty closed bounded convex subset of a Banach
		space.  If $T:\mathcal K\to\mathcal K$ is continuous and compact,
		then $T$ has a fixed point in $\mathcal K$.
	\end{lemma}

	\section{The no-swirl problem}\label{sec:No-Swirl}
	In this section we set $Z\equiv0$ in \eqref{eqn:halfD1} and analyze
	the resulting no-swirl problem.
	
	\subsection{Reduction for swirl-free external forces}\label{sec:reductNo-swirl}
	
	If the force has no swirl, then every bounded-profile solution also
	has no swirl.  Indeed, set $H=h'+h\cot\varphi$.  The third equation
	of \eqref{eqn:halfD1} gives $H'=gH$, so either $H\equiv0$ or $H$ has
	one strict sign.  On the other hand,
	\[
	\int_0^{\pi/2}H\sin\varphi\,\mathrm d\varphi
	=h\sin\varphi\big|_0^{\pi/2}=0.
	\]
	Thus $H=0$, and $(h\sin\varphi)'=0$ gives $h=0$; see also~\cite[Theorem~8.1]{Tsai18}.
	
	With $h=0$, the quantities in
	Section~\ref{sec:integrated-reconstruction} become
	\begin{equation}\label{eqn:poloidal-data}
		G(\varphi)=X(\varphi)+2\int_0^\varphi Y(\tau)\,\mathrm d\tau,
		\qquad
		\mcC=-\int_0^{\pi/2}G(\tau)\sin\tau\cos\tau\,\mathrm d\tau.
	\end{equation}
	The reduced problem is precisely
	\[
	\frac{\rmd M}{\rmd t} + \frac12M^2=N(t), \quad t\in(0,1),\qquad 
	M(0)=0,
	\]
	where $N(t)=-\frac{1}{(1-t^2)^2}\int_t^1J(s)\,\mathrm ds$ and $J$ is defined in Section~\ref{sec:integrated-reconstruction}.  By
	Lemma~\ref{lm:reconstruction}, the existence of a bounded-profile
	solution is equivalent to the existence of this Riccati solution on
	$[0,1]$.  When it exists, it is unique.
	
	\subsection{Estimate for swirl-free forces}\label{sec:EUNo-swirl}
	
	We first give a sufficient condition for the Riccati equation to have
	a solution on $[0,1]$, together with estimates for that solution.
	
	\begin{lemma}
		\label{lm:Riccati}
		For $N\in C^{0}[0,1]$,
		consider the initial value problem
		\[
		\frac{\mathrm{d}M}{\mathrm{d}t} + \frac12 M^{2} = N
		\quad \text{on } [0,1],\qquad M(0)=0.
		\]
		(i) If $\|N\|_{C^{0}[0,1]} < \dfrac12$, then there exists a unique solution $M\in C^{1}[0,1]$ satisfying
		\[
		\|M\|_{C^{0}[0,1]} \le 2\|N\|_{C^{0}[0,1]},\qquad
		\|M\|_{C^{1}[0,1]} \le 4\|N\|_{C^{0}[0,1]}.
		\]
		(ii) For any $N_{1},N_{2}$ with
		$\|N_{i}\|_{C^{0}[0,1]} \le \delta < \dfrac12$, let $M_{1},M_{2}$
		be the corresponding solutions. Then
		\[
		\|M_{1}-M_{2}\|_{C^{1}[0,1]}
		\le \frac{2}{1-2\delta}\,
		\|N_{1}-N_{2}\|_{C^{0}[0,1]}.
		\]
	\end{lemma}
	
	\begin{proof}
		(i) If $N\equiv0$, the unique solution is $M=0$.  We therefore
		assume $\|N\|_{C^0}>0$.  For brevity we omit the interval
		$[0,1]$ in the norm notation.
		Define the operator $\Psi$ on $C^{0}[0,1]$ by
		\[
		\Psi M(t) := \int_{0}^{t} \Bigl( N(s) - \frac12 M^{2}(s) \Bigr) \, \mathrm{d}s ,\qquad t\in[0,1].
		\]
		For any $M\in C^{0}[0,1]$,
		\[
		\|\Psi M\|_{C^{0}} \le \|N\|_{C^{0}} + \frac12 \|M\|_{C^{0}}^{2}.
		\]
		Let $\mathfrak{X}_{R} = \{\, M\in C^{0}[0,1] : M(0)=0,\ \|M\|_{C^{0}} \le R \,\}$ with $R>0$ to be chosen.
		If $M\in\mathfrak{X}_{R}$, then
		\[
		\|\Psi M\|_{C^{0}} \le \|N\|_{C^{0}} + \frac12 R^{2}.
		\]
		Choosing $R$ to satisfy $\|N\|_{C^{0}} + \frac12R^{2} = R$ gives $R = 1 - \sqrt{\,1 - 2\|N\|_{C^{0}}\,}$,
		which is well defined and positive because $\|N\|_{C^{0}} < \frac12$.
		With this $R$, $\Psi$ maps $\mathfrak{X}_{R}$ into itself.
		
		Now estimate the difference of two images:
		\[
		(\Psi M_{1} - \Psi M_{2})(t) = \frac12 \int_{0}^{t} \bigl( M_{2}^{2} - M_{1}^{2} \bigr) \, \mathrm{d}s,
		\]
		hence
		\[
		\|\Psi M_{1} - \Psi M_{2}\|_{C^{0}}
		\le \frac12 \bigl( \|M_{1}\|_{C^{0}} + \|M_{2}\|_{C^{0}} \bigr) \|M_{1}-M_{2}\|_{C^{0}}
		\le R \,\|M_{1}-M_{2}\|_{C^{0}}.
		\]
		Because $R = 1 - \sqrt{\,1-2\|N\|_{C^{0}}\,} < 1$, the map $\Psi$ is a contraction on $\mathfrak{X}_{R}$. Consequently, there exists a unique $M\in\mathfrak{X}_{R}$ with $M=\Psi M$.  
		Since $N(t)-\frac12M^2$ is locally Lipschitz with respect to $M$, this solution is unique among all solutions of the initial-value
		problem.
		
		From $M \in \mathfrak{X}_R$ we get $\|M\|_{C^{0}}\le R$, and the differential equation yields
		\[
		\Bigl\|\frac{\mathrm{d}M}{\mathrm{d}t}\Bigr\|_{C^{0}} \le \|N\|_{C^{0}} + \frac12 R^{2} = R
		\]
		by the choice of $R$.  Thus $\|M\|_{C^{1}} \le 2R$.
		Using $1-\sqrt{1-x}\le x$ for $x\in[0,1)$, we obtain
		$R\le 2\|N\|_{C^{0}}$, and therefore
		\[
		\|M\|_{C^{0}} \le 2\|N\|_{C^{0}}, \qquad
		\|M\|_{C^{1}} \le 4\|N\|_{C^{0}} .
		\]
		This proves (i).
		
		\noindent(ii) From the integral formulation $M_i = \int_0^t(N_i - \frac12 M_i^2)$, and the bound
		$\|M_i\|_{C^0}\le 2\delta$,
		\[
		\begin{aligned}
			\|M_1-M_2\|_{C^0}
			\le \|N_1-N_2\|_{C^0} + 2\delta\,\|M_1-M_2\|_{C^0}.
		\end{aligned}
		\]
		Hence
		\[
		\|M_1-M_2\|_{C^0} \le \frac{1}{1-2\delta}\,\|N_1-N_2\|_{C^0}.
		\]
		For the derivatives, $\frac{\mathrm{d}M_i}{\mathrm{d}t} = N_i - \frac12 M_i^2$, so using the same bound on $\|M_i\|_{C^0}$,
		\[
		\begin{aligned}
			\Bigl\|\frac{\mathrm{d}M_1}{\mathrm{d}t}-\frac{\mathrm{d}M_2}{\mathrm{d}t}\Bigr\|_{C^0}
			\le \|N_1-N_2\|_{C^0} + 2\delta\|M_1-M_2\|_{C^0} 
			\le \frac{1}{1-2\delta}\,\|N_1-N_2\|_{C^0}.
		\end{aligned}
		\]
		Adding the two estimates gives the required $C^1$ bound.
	\end{proof}
	
	The following proposition shows that a swirl-free force with small curl has a unique solution with small velocity profiles.
	
	\begin{proposition}
		\label{prop:small-poloidal}
		Let the swirl-free force satisfy \eqref{eqn:force-regularity} in the
		half-space, and set $\zeta=\|X'+2Y\|_{C^0[0,\pi/2]}$.
		If $\zeta<2/\pi$, there is a unique bounded-profile solution, and
		\begin{equation}\label{eqn:small-poloidal-bounds}
			\|N\|_{C^0[0,1]}\leq\frac\pi4\zeta,
			\qquad
			\|M\|_{C^0[0,1]}\leq\frac\pi2\zeta,
			\qquad
			\|M\|_{C^1[0,1]}\leq\pi\zeta.
		\end{equation}
		Moreover, $h=0$ and
		$\|f\|_{C^0[0,\pi/2]} + \|g\|_{C^1[0,\pi/2]}\leq C\zeta$ for an absolute constant $C$.
	\end{proposition}
	
	\begin{proof}
		First, \eqref{eqn:poloidal-data} gives
		\begin{equation}\label{eqn:Gand2C}
			G(\varphi)+2\mcC
			=2\int_0^{\pi/2}
			\bigl(G(\varphi)-G(\tau)\bigr)\sin\tau\cos\tau\,\mathrm d\tau.
		\end{equation}
		Since $G'=X'+2Y$ and $2\int_0^{\pi/2}\sin\tau\cos\tau\,\mathrm d\tau=1$,
		\[
		\|G+2\mcC\|_{C^0}\leq\frac\pi2\zeta.
		\]
		Using $J(1)=0$, the numerator of $N$ becomes
		\begin{equation*}
			\begin{aligned}
				- \int_t^1 J(s)\, \mathrm{d}s
				&= t J(t) + \int_t^1 s \,\frac{\mathrm{d}J}{\mathrm{d}s}(s) \, \mathrm{d}s
				= \int_t^1 (s-t)\, \frac{\mathrm{d}J}{\mathrm{d}s}(s)\, \mathrm{d}s \\
				&= (1-t)^2 \int_0^1 \frac{\mathrm{d}J}{\mathrm{d}s}\bigl( 1-\sigma(1-t) \bigr)(1-\sigma) \, \mathrm{d}\sigma ,
			\end{aligned}
		\end{equation*}
		where the last equality follows from $s=1-\sigma(1-t)$.
		Therefore
		\begin{align*}
			N(t) = \frac{1}{(1+t)^2}\int_0^1 \frac{\mathrm{d}J}{\mathrm{d}s}\bigl( 1-\sigma(1-t)\bigr)(1-\sigma) \, \rmd\sigma.
		\end{align*}
		Using $\frac{\rmd J}{\rmd t}=-(G+2\mcC)$, the above equation gives
		\[
		\|N\|_{C^0[0,1]}
		\leq\frac12\|G+2\mcC\|_{C^0[0,\pi/2]}
		\leq\frac\pi4\zeta.
		\]
		If $\zeta<2/\pi$, Lemma~\ref{lm:Riccati} yields the two estimates for $M$ in
		\eqref{eqn:small-poloidal-bounds}.  Lemma~\ref{lm:reconstruction}
		recovers a solution in the stated class.  The no-swirl reduction gives uniqueness in the bounded-profile class.  Finally,
		\[
		g(\varphi)=M(t)\sin\varphi,
		\quad g'(\varphi)=tM(t)-(1-t^2)\frac{\rmd M}{\rmd t}(t),
		\quad f(\varphi)=-2tM(t)+(1-t^2)\frac{\rmd M}{\rmd t}(t),
		\]
		which gives the remaining estimates directly.
	\end{proof}
	
	\subsection{The interval of force amplitudes}\label{sec:blowup}
	
	This subsection is devoted to the proof of Theorem~\ref{thm:main}(i). Recall from Sections~\ref{sec:integrated-reconstruction} and~\ref{sec:reductNo-swirl} that the no-swirl problem is equivalent to the Riccati equation
	\begin{equation*}
		\begin{aligned}
			\frac{\mathrm{d}M}{\mathrm{d}t} + \frac12 M^{2} = N \quad \text{on } [0,1],\qquad
			M(0)=0,
		\end{aligned}
	\end{equation*}
	where
	\[
	N(t)=-\frac{1}{(1-t^{2})^{2}}\int_{t}^{1} J(s)\,\mathrm{d}s,\qquad
	J(t)=\int_{0}^{\varphi} \bigl(G(\tau)+2\mcC\bigr)\sin\tau\,\mathrm{d}\tau,
	\]
	with $\mcC=-\int_{0}^{\frac{\pi}{2}} G(\tau) \sin\tau \cos\tau \, \mathrm{d}\tau$ and $t=\cos\varphi$.
	Replacing $\BF$ by $\lambda\BF$ changes $G$ to $\lambda G$ and $N$ to $\lambda N$, yielding
	\begin{equation}\label{eqn:sclRiccati}
		\left\{
		\begin{aligned}
			\frac{\mathrm{d}M}{\mathrm{d}t} + \frac12 M^{2} &= \lambda N, \quad t \in [0,1],\\
			M(0)&=0.
		\end{aligned}
		\right.
	\end{equation}
	
	To discuss solutions beyond $t=1$, we extend $N$ constantly to $(-1,2)$ by $N(t)=N(0)$ for $t\in(-1,0]$ and $N(t)=N(1)$ for $t\in[1,2)$.
	
	\begin{theorem}\label{thm:Auxiliary}
		Let $N\in C^{0}[0,1]$ and consider the initial value problem \eqref{eqn:sclRiccati} on $(-1,2)$.
		For each $\lambda$, let $M_\lambda$ be the unique maximal solution
		of the initial-value problem, and let $[0,\omega_\lambda)$ denote
		its maximal right interval inside $(-1,2)$. Define
		\[
		I:=\{\lambda\in\mathbb{R}\mid\omega_{\lambda}>1\}.
		\]
		
		Then $I$ is an open interval containing $0$; write $I=(\underline{\lambda},\overline{\lambda})$.
		Moreover, the following classification holds.
		\begin{enumerate}[(i)]
			\item If $N\le0$ on $(0,1)$ and $N\not\equiv0$, then $\underline{\lambda}=-\infty$ and $\overline{\lambda}\in(0,+\infty)$;
			\item If $N\ge0$ on $(0,1)$ and $N\not\equiv0$, then $\underline{\lambda}\in(-\infty,0)$ and $\overline{\lambda}=+\infty$;
			\item If $N$ changes sign on $(0,1)$, then $\underline{\lambda}\in(-\infty,0)$ and $\overline{\lambda}\in(0,+\infty)$;
			\item If $N\equiv0$ on $(0,1)$, then $\underline{\lambda}=-\infty$ and $\overline{\lambda}=+\infty$.
		\end{enumerate}
		If $\underline{\lambda}$ (resp.\ $\overline{\lambda}$) is finite, then $\omega_{\underline{\lambda}}=1$ (resp.\ $\omega_{\overline{\lambda}}=1$); in other words, the corresponding solution exists on $[0,1)$ and blows down to $-\infty$ exactly at $t=1$.
	\end{theorem}
	
	We prove the theorem with the help of three lemmas.
	
	\begin{lemma}\label{lm:lm1}
		Consider the initial value problem \eqref{eqn:sclRiccati}.
		If $N$ is negative somewhere in $(0,1)$, then there exists $\lambda^{*}>0$ such that $\omega_{\lambda}\le 1$ for all $\lambda\ge\lambda^{*}$.
		Similarly, if $N$ is positive somewhere in $(0,1)$, then there exists $\lambda_{*}<0$ such that $\omega_{\lambda}\le 1$ for all $\lambda\le\lambda_{*}$.
	\end{lemma}
	
	\begin{proof}
		We prove the first statement; the second follows by symmetry.
		Choose $t_{1}\in(0,1)$ and $C>0$ with $N(t_{1})=-2C$. By continuity of $N$, there exists $\delta_{1}>0$ such that $t_{1}-\delta_{1}>0$ and $N(t)<-C$ on $(t_{1}-\delta_{1},t_{1})$.
		If the solution does not exist up to $t_{1}-\delta_{1}$, then $\omega_{\lambda}\le t_{1}-\delta_{1}<1$ and the claim holds trivially.
		Hence we may assume $\omega_{\lambda}>t_{1}-\delta_{1}$, so that $M(t_{1}-\delta_{1})$ is finite. On the interval $(t_{1}-\delta_{1}, t_{1})$, for any $\lambda>0$ we have
		\[
		\frac{\mathrm{d}M}{\mathrm{d}t} \le -\Bigl(\frac{1}{2}M^{2}+\lambda C\Bigr),
		\]
		which gives
		\[
		\frac{\mathrm{d}}{\mathrm{d}t}\Bigl( \arctan \Bigl(\frac{M(t)}{\sqrt{2\lambda C}}\Bigr) \Bigr)
		\le -\sqrt{\frac{\lambda C}{2}}.
		\]
		Integrating from $t_{1}-\delta_{1}$ to $t$ yields
		\[
		\arctan \Bigl(\frac{M(t)}{\sqrt{2\lambda C}}\Bigr)
		\le -\sqrt{\frac{\lambda C}{2}}\bigl(t-\tilde{t}_{1}\bigr),
		\]
		where $\tilde{t}_{1} = t_{1}-\delta_{1}+\frac{2}{\sqrt{2\lambda C}}\arctan\bigl(\frac{M(t_{1}-\delta_{1})}{\sqrt{2\lambda C}}\bigr)$.
		Hence
		\[
		M(t) \le -\sqrt{2\lambda C}\, \tan\Bigl(\sqrt{\frac{\lambda C}{2}}(t-\tilde{t}_{1})\Bigr).
		\]
		Thus $M$ blows down to $-\infty$ no later than $t = \frac{\pi}{\sqrt{2\lambda C}} + \tilde{t}_{1}$.
		If $\lambda \ge \lambda^{*} := \frac{2}{C}\bigl(\frac{\pi}{\delta_{1}}\bigr)^{2}$, then
		\[
		\frac{\pi}{\sqrt{2\lambda C}} + \tilde{t}_{1}
		\le \frac{2\pi}{\sqrt{2\lambda C}} + t_{1} - \delta_{1} \le t_{1} < 1,
		\]
		so $\omega_{\lambda} \le t_{1} < 1$. This completes the proof.
	\end{proof}

	\begin{remark}\label{rmk:rough-Riccati-bound}
		The proof gives rough bounds for the endpoints of
		$I=(\underline\lambda,\overline\lambda)$. If $N(t)\leq-C$
		on $(t_0-\delta,t_0)\subset(0,1)$, 
		then
		\[
		\overline\lambda \leq \frac{2}{C}\left(\frac{\pi}{\delta}\right)^2.
		\]
		Similarly, if $N(t)\geq C$ on the same interval, then $\underline\lambda \geq -\frac{2}{C}\left(\frac{\pi}{\delta}\right)^2$.
		These bounds are generally not sharp.
	\end{remark}

	\begin{lemma}\label{lm:open}
		The set $I = \{\lambda \in \mathbb{R} \mid \omega_{\lambda} > 1\}$ is an open interval containing $0$. 
	\end{lemma}
	
	\begin{proof}
		\underline{Step 1.} \emph{Openness.}
		If $\lambda=0$, equation \eqref{eqn:sclRiccati} becomes $\frac{\mathrm{d}M}{\mathrm{d}t} + \frac12 M^{2} = 0$ with $M(0)=0$, whose unique solution is $M_{0}\equiv0$. By the constant extension of $N$, this solution exists on $(-1,2)$; hence $\omega_{0} = 2 > 1$, so $0\in I$ and $I\not =\emptyset$.
		Openness follows from the lower semicontinuity of $\omega_{\lambda}$ with respect to $\lambda$ (see \cite[Chapter~V]{Hartman02}). Indeed, for $\lambda_{1}\in I$ we have $\omega_{\lambda_{1}}>1$; choose $\delta_2>0$ such that $\omega_{\lambda_{1}}>1+\delta_2$. By lower semicontinuity there exists $\varepsilon>0$ such that
		$\omega_{\lambda} \ge \omega_{\lambda_{1}} - \frac{\delta_2}{2} > 1 + \frac{\delta_2}{2} > 1$ whenever $|\lambda-\lambda_{1}|<\varepsilon$. Thus $(\lambda_{1}-\varepsilon,\lambda_{1}+\varepsilon)\subset I$, proving openness.
		
		\underline{Step 2.} \emph{Interval property.}
		Since $I$ is open and contains $0$, it is a union of disjoint open intervals; let $(\underline{\lambda},\overline{\lambda})$ be the connected component containing $0$. We show that actually $I = (\underline{\lambda},\overline{\lambda})$. Denote $C_{1} := \max\limits_{t\in[0,1]}|N(t)|$.
		For every $\lambda$, irrespective of the sign of $N$, scalar
		comparison with $\frac{\rmd v}{\rmd t}=|\lambda|C_1-v^2/2$ gives
		\begin{equation}\label{eqn:universal-Riccati-upper}
			M_\lambda(t)\leq\sqrt{2|\lambda|C_1}
		\end{equation}
		on its right existence interval.  In particular, any finite-time
		singularity inside $(-1,2)$ is a blow-down to $-\infty$.
		
		\underline{Case 1.} \emph{$N \le 0$ on $(0,1)$ and $N \not\equiv 0$.}
		Lemma~\ref{lm:lm1} forces $0<\overline{\lambda}<+\infty$ and the openness of $I$ gives $\omega_{\overline{\lambda}}\le 1$.
		We prove that $\underline{\lambda}=-\infty$ and that $\omega_{\lambda}\le 1$ for all $\lambda>\overline{\lambda}$.
		Since $\omega_{\lambda}>1$ for $\lambda\in[0,\overline{\lambda})$, it suffices to prove the same inequality for all $\lambda<0$; this will force $\underline{\lambda}=-\infty$.
		To this end, fix any $\lambda<0$ and choose a $\lambda_2 \in (0,\overline{\lambda})$; then $\omega_{\lambda_2}>1$.
		Since $N\le0$ and $\lambda<\lambda_{2}$, we have
		\[
		\lambda_{2} N - \frac12 M^{2} \le \lambda N - \frac12 M^{2} \le |\lambda| C_{1} - \frac12 M^{2}.
		\]
		By standard comparison theorems (see \cite[Chapter~III]{Hartman02}),
		\begin{align}\label{eqn:lm2Comp}
			M_{\lambda_{2}}(t) \le M_{\lambda}(t) \le \sqrt{2|\lambda| C_{1}}.
		\end{align}
		Thus $\omega_{\lambda} \ge \omega_{\lambda_{2}} > 1$ for every $\lambda<0$, forcing $\underline{\lambda}=-\infty$.
		Moreover, the upper bound in \eqref{eqn:lm2Comp} shows that any finite-time singularity is a blow-down to $-\infty$.
		Since $\omega_{\overline{\lambda}} \le 1$, $M_{\overline{\lambda}}$ blows down no later than $t=1$.
		
		Now take any $\lambda > \overline{\lambda}>0$ and define $v(t) = \lambda\frac{M_{\overline{\lambda}}(t)}{\overline{\lambda}}$.
		A direct computation gives
		\[
		\frac{\mathrm{d}v}{\mathrm{d}t} = \lambda N - \frac12 \frac{\overline{\lambda}}{\lambda} v^{2}
		\ge \lambda N - \frac12 v^{2},
		\]
		where we used $0<\overline{\lambda}/\lambda<1$. Since $M_{\lambda}(0)=v(0)=0$, comparison gives $M_{\lambda}\le v$. Because $v$ blows down no later than $t=1$, we get $\omega_{\lambda}\le1$ for all $\lambda>\overline{\lambda}$.
		
		\underline{Case 2.} \emph{$N \ge 0$ on $(0,1)$ and $N \not\equiv 0$.}
		By symmetry, Lemma~\ref{lm:lm1} gives $-\infty<\underline{\lambda}<0$, openness yields $\omega_{\underline{\lambda}}\le1$, and the same comparison argument shows $\overline{\lambda}=+\infty$ and $\omega_{\lambda}\le1$ for all $\lambda<\underline{\lambda}$. We omit the details.
		
		\underline{Case 3.} \emph{$N$ changes sign on $(0,1)$.}
		Lemma~\ref{lm:lm1} implies $-\infty<\underline{\lambda}<0<\overline{\lambda}<+\infty$, together with $\omega_{\underline{\lambda}}\le1$ and $\omega_{\overline{\lambda}}\le1$.
		We must show $\omega_{\lambda}\le1$ for $\lambda\notin[\underline{\lambda},\overline{\lambda}]$.
		By \eqref{eqn:universal-Riccati-upper}, any finite-time blow-up is a blow-down to $-\infty$, and $M_{\underline{\lambda}},M_{\overline{\lambda}}$ blow down no later than $t=1$.
		For $\lambda > \overline{\lambda}$, the scaling $v(t)=\lambda\frac{M_{\overline{\lambda}}(t)}{\overline{\lambda}}$ again satisfies $\frac{\mathrm{d}v}{\mathrm{d}t}\ge \lambda N - \frac12 v^{2}$, hence $M_{\lambda}\le v$ and $\omega_{\lambda}\le1$.
		For $\lambda < \underline{\lambda}$, setting $v(t)=\lambda\frac{M_{\underline{\lambda}}(t)}{\underline{\lambda}}$ yields $\frac{\mathrm{d}v}{\mathrm{d}t}\ge \lambda N - \frac12 v^{2}$, since $0<\underline{\lambda}/\lambda<1$; therefore $M_{\lambda}\le v$ and $\omega_{\lambda}\le1$.
		
		If $N\equiv0$, then $M_\lambda=0$ for every $\lambda$, and
		$I=\mathbb R$ directly.
	\end{proof}
	
	\begin{lemma}\label{lm:exactblow}
		If $\underline{\lambda}$ (resp.\ $\overline{\lambda}$) is finite, then $\omega_{\underline{\lambda}}=1$ (resp.\ $\omega_{\overline{\lambda}}=1$).
	\end{lemma}
	
	\begin{proof}
		We prove $\omega_{\overline{\lambda}}=1$ by contradiction; the case for $\underline{\lambda}$ is analogous.
		Suppose $\omega_{\overline{\lambda}}<1$. Choose $t_{R}$ with $t_{R}<\omega_{\overline{\lambda}}<1$ and a number $R>0$ to be specified later such that $M_{\overline{\lambda}}(t_{R})<-2R$.
		By continuous dependence on the parameter (see \cite[Chapter~V]{Hartman02}), there exists $\lambda_{3}\in(0,\overline{\lambda})$ close to $\overline{\lambda}$ with $M_{\lambda_{3}}(t_{R}) < -R$.
		Let $C_{1} := \max\limits_{t\in[0,1]}|N(t)|$. Then
		\[
		\frac{\mathrm{d}M_{\lambda_{3}}}{\mathrm{d}t} = \lambda_{3} N - \frac12 M_{\lambda_{3}}^{2} \le \overline{\lambda} C_{1} - \frac12 M_{\lambda_{3}}^{2}.
		\]
		Consider the auxiliary problem
		\[
		\frac{\mathrm{d}v}{\mathrm{d}t} = \overline{\lambda} C_{1} - \frac12 v^{2},\qquad v(t_{R}) = -R.
		\]
		Choose $R > \sqrt{4\overline{\lambda} C_{1}}$. Then $\frac{\mathrm{d}v}{\mathrm{d}t}(t_{R})<0$, and one checks that $v(t)\le -R$, $\frac{\mathrm{d}v}{\mathrm{d}t}(t)<0$ for all $t\ge t_{R}$ as long as the solution exists.
		Moreover, for $v\le -R$ we have $\frac14 v^{2} \ge \overline{\lambda} C_{1}$, hence
		\[
		\frac{\mathrm{d}v}{\mathrm{d}t} = \overline{\lambda} C_{1} - \frac12 v^{2} \le \frac14 v^{2} - \frac12 v^{2} = -\frac14 v^{2}.
		\]
		Integrating from $t_{R}$ to $t$ gives
		\[
		-\frac{1}{v(t)} \le \frac{1}{R} - \frac14 (t - t_{R}),
		\]
		
		and because $v(t)<0$, the right-hand side must stay positive, which forces
		$t < t_{R} + \frac{4}{R}$.
		Thus $v$ blows down to $-\infty$ no later than $t_{R} + 4/R$.
		Choosing $R > \max\bigl\{ \sqrt{4\overline{\lambda} C_{1}},\; \frac{4}{1-\omega_{\overline{\lambda}}} \bigr\}$ guarantees $t_{R} + 4/R < 1$.
		By comparison, $M_{\lambda_{3}}(t)\le v(t)$ on their common interval, so $M_{\lambda_{3}}$ also blows down before $t=1$, contradicting $\omega_{\lambda_{3}}>1$. Hence $\omega_{\overline{\lambda}}=1$.
	\end{proof}
	
	
	The following lemma describes the asymptotic behavior of the
	no-swirl solution at a finite endpoint of $I$ and gives its
	singularity at $t=1$ $(\varphi=0)$.
	
	\begin{lemma}[The velocity at a finite endpoint]
		\label{lm:endpoint-asymptotics}
		Suppose $N\in C^0[0,1]$, and let
		$\lambda\in \{\underline\lambda,\overline\lambda\}\cap\R$ 
		be a finite endpoint of $I$. Then the corresponding solution $M_\lambda$ of \eqref{eqn:sclRiccati} satisfies
		\begin{equation}\label{eqn:critical-M-asymptotic}
			M_\lambda(t)=-\frac{2}{1-t}+O(1-t),\qquad t\to1^-.
		\end{equation}
		The corresponding polar velocity profile satisfies
		\begin{equation}\label{eqn:critical-g-asymptotic}
			g_\lambda(\varphi)
			=
			-\frac4\varphi+O(\varphi),
			\qquad \varphi\to0^+.
		\end{equation}
		In particular, the velocity profile is unbounded.
	\end{lemma}
	
	\begin{proof}
		Since $M_\lambda(t)\to-\infty$ as $t\to1^-$, set
		$v(t)=-2/M_\lambda(t)$ near $t=1$ and extend it continuously by
		$v(1)=0$. Then $v>0$, and the Riccati equation becomes
		\[
		\frac{\rmd v}{\rmd t}
		=
		-1+\frac{\lambda}{2}N(t)v^2.
		\]
		Since $N$ is bounded and $v(t)\to0$, the equation gives
		$\frac{\rmd v}{\rmd t}\to-1$. 
		Combining this with $v(1)=0$, we get
		$v(t)=1-t+o(1-t)$. In particular, $v(t)=O(1-t)$. Substituting this
		back into the equation gives
		$\frac{\rmd v}{\rmd t}=-1+O((1-t)^2)$. Integrating gives
		$v(t)=(1-t)\big(1 + O((1-t)^2)\big)$. Therefore,
		\[
		M_\lambda(t)
		=
		-\frac{2}{v(t)}
		=
		-\frac{2}{1-t}+O(1-t),
		\]
		which proves \eqref{eqn:critical-M-asymptotic}.
		Since $g_\lambda(\varphi)=M_\lambda(\cos\varphi)\sin\varphi$,
		\[
		g_\lambda(\varphi)
		=
		-2\frac{\sin\varphi}{1-\cos\varphi}+O(\varphi^3)
		=
		-2\cot\frac{\varphi}{2}+O(\varphi^3)
		=
		-\frac4\varphi+O(\varphi).
		\]
		Thus \eqref{eqn:critical-g-asymptotic} holds, and the velocity profile
		is unbounded.
	\end{proof}
	
	
	\begin{proof}[\textbf{Proof of Theorem~\ref{thm:main}(i)}]
		By the no-swirl reduction and Lemma~\ref{lm:reconstruction},
		the Navier--Stokes problem with force $\BF_{\lambda,0}$ has a
		bounded-profile solution if and only if the corresponding Riccati
		solution extends to $[0,1]$. When it exists, the solution is unique
		and has no swirl, and it has the stated regularity. Hence the
		admissible set is the interval given by
		Theorem~\ref{thm:Auxiliary}.
		
		It remains to compare the signs of $N$ and $\mathcal N$. Since
		\[
		N(t)
		=
		-\frac{1}{(1-t^2)^2}\int_t^1J(s)\,\rmd s,
		\]
		we have $\mathrm{sgn}\, N=\mathrm{sgn}\, (-\int_{t}^{1} J(s)\, \mathrm{d}s)$.
		Using the change of variable $s=\cos\tau$, we compute
		\begin{equation}\label{eqn:PfThmB}
			\begin{aligned}
				-\int_{t}^{1} J(s)\, \mathrm{d}s
				&= -\int_{\varphi}^{0} \big( \int_0^{\tau} (G(\sigma) + 2\mcC) \sin \sigma \, \mathrm{d}\sigma \big) \, \mathrm{d}(\cos\tau)\\
				&= \cos \varphi \int_0^{\vp} (G(\tau) + 2\mcC) \sin \tau \, \mathrm{d}\tau  - \int_{0}^{\varphi} (G(\tau)+2\mcC)\sin \tau \cos \tau \, \mathrm{d}\tau\\
				&= \int_{0}^{\varphi} (G(\tau)+2\mcC) \sin \tau (\cos \varphi - \cos \tau) \, \mathrm{d}\tau = \mathcal{N}(\varphi).
			\end{aligned}
		\end{equation}
		Thus $N(t)$ and $\mathcal N(\varphi)$ have the same sign. The
		classification in Theorem~\ref{thm:Auxiliary} therefore gives
		exactly the four cases in Theorem~\ref{thm:main}(i).
	\end{proof}
	
	\section{Existence for arbitrary swirl forcing}\label{sec:withswirl}
	
	In this section we prove Theorem~\ref{thm:main}(ii) and its corollaries.
	With swirl forcing, the Riccati equation and the swirl equation are
	coupled, so direct estimates for either equation are not enough to
	control the whole system. The main point is that the swirl part
	produces a nonnegative term $Q[h]$ in the Riccati equation. This gives the lower bound $M\geq M_\lambda$, which leads to closed estimates for $h$ and $Q[h]$ without using an upper bound for $M$. 
	We can then apply the Schauder fixed-point theorem to prove existence for arbitrary swirl forcing. 
	When the swirl forcing is small, difference estimates also give uniqueness.
	
	\subsection{The quadratic contribution of swirl}\label{sec:reduct-swirl}
	
	This subsection isolates the swirl contribution $Q[h]$ to the Riccati forcing, proves its nonnegativity, and establishes quadratic estimates on $Q[h]$ needed in the fixed-point and uniqueness arguments.
	
	For the two-parameter force \eqref{eqn:two-parameter-force}, Section~\ref{sec:integrated-reconstruction} applies with force profiles $(\lambda X,\lambda Y,\mu Z)$. This yields the quantity $G$ defined in \eqref{eqn:general-G}. We split it as $G=\lambda G_0+G_h$,
	where
	\[
	G_0(\varphi)=X(\varphi)+2\int_0^\varphi Y(\tau)\,\mathrm d\tau
	\]
	depends only on $X$ and $Y$, and corresponds to the no-swirl case ($\lambda=1$, $h=0$), while
	\[
	G_h(\varphi)=h(\varphi)^2+2\int_0^\varphi h(\tau)^2\cot\tau\,\mathrm d\tau
	\]
	depends only on the swirl profile $h$.
	The dependence of the Riccati forcing on $G$ can be summarized as
	\[
	G
	\xrightarrow{\eqref{eqn:halfspace-C}}
	\mcC
	\xrightarrow[\text{together with }G]{\eqref{eqn:def-J}}
	J
	\xrightarrow{\eqref{eqn:def-M-N}}
	N.
	\]
	Each of the three maps in the diagram above is linear. 
	Therefore, the no-swirl and swirl parts separate, and the Riccati forcing decomposes as
	\begin{equation}\label{eqn:N-decomposition}
		N=\lambda N_0+Q[h],
	\end{equation}
	where $N_0$ is the no-swirl forcing and $Q[h]$ is the contribution generated by $G_h$.
	
	We now define $Q[h]$ explicitly. First,
	\[
	\mcC_h
	=-\int_0^{\pi/2}G_h(\varphi)\sin\varphi\cos\varphi\,\mathrm d\varphi.
	\]
	For $t=\cos\varphi$, set
	\[
	J_h(t)=\int_0^\varphi\bigl(G_h(\tau)+2\mcC_h\bigr)\sin\tau\,\mathrm d\tau,
	\qquad
	Q[h](t)=-\frac{1}{(1-t^2)^2}\int_t^1J_h(s)\,\mathrm ds.
	\]
	Thus the coupled system becomes
	\[
	\frac{\rmd M}{\rmd t} + \frac12M^2=\lambda N_0+Q[h],
	\qquad M(0)=0,
	\]
	and
	\begin{equation}\label{eqn:swirlh}
		(h'+h\cot\varphi)'=g(h'+h\cot\varphi)-\mu Z,
		\qquad h(0)=h(\pi/2)=0,
	\end{equation}
	with $g(\varphi)=M(\cos\varphi)\sin\varphi$.
	The first lemma shows that this swirl contribution is nonnegative.
	
	\begin{lemma}[Positivity of the swirl contribution]
		\label{lm:positive-swirl-contribution}
		Let $h\in C^1[0,\pi/2]$ satisfy $h(0)=h(\pi/2)=0$, and let $G_h$, $\mcC_h$, $J_h$, and $Q[h](t)$ be defined as above. Then $Q[h]$ extends continuously to $[0,1]$ and, for $0\leq t<1$,
		\begin{align}\label{eqn:halfspace-positive-swirl}
			\begin{aligned}
				Q[h](t)
				=
				\frac{1}{(1+t)^2}
				\int_0^t\frac{y}{1-y^2}
				h(\arccos y)^2\,\mathrm dy 
				+ \frac{t}{(1-t^2)^2}
				\int_t^1
				\frac{1-y}{1+y}
				h(\arccos y)^2\,\mathrm dy.
			\end{aligned}
		\end{align}
		In particular, $Q[h]\geq0$ on $[0,1]$. Moreover, $Q[h](t)>0$ for every $t\in(0,1]$ whenever $h\not\equiv0$.
	\end{lemma}
	
	\begin{proof}
		For $x\in[0,1]$, changing variables
		$y=\cos\tau$ in the definition of $G_h$, we obtain
		\[
		G_h(\arccos x)
		=
		h(\arccos x)^2
		+
		2\int_0^{\arccos x}
		h(\tau)^2\cot\tau\,\mathrm d\tau
		= h(\arccos x)^2
		+
		2\int_x^1
		\frac{y\,h(\arccos y)^2}{1-y^2}\,\mathrm dy.
		\]
		Consequently, for every $a\in[0,1]$, Fubini's theorem yields
		\begin{align}\label{eqn:weighted-Gh}
			\begin{aligned}
				\int_a^1(x-a)G_h(\arccos x)\,\mathrm dx
				&=
				\int_a^1(x-a)h(\arccos x)^2\,\mathrm dx
				+
				2\int_a^1 \frac{y\,h(\arccos y)^2}{1-y^2}\int_a^y (x-a)
				\,\mathrm dx\,\mathrm dy\\
				&=
				\int_a^1
				\frac{(y-a)(1-ay)}{1-y^2}
				h(\arccos y)^2\,\mathrm dy.
			\end{aligned}
		\end{align}
		Changing variables in the definition of $\mcC_h$ and taking $a=0$ in \eqref{eqn:weighted-Gh} gives
		\begin{equation}
			\label{eqn:halfspace-Ch}
			\mathcal C_h
			=
			-\int_0^1 xG_h(\arccos x)\,\mathrm dx
			=-\int_0^1\frac{y}{1-y^2}
			h(\arccos y)^2\,\mathrm dy.
		\end{equation}
		Next, changing variables in the definition of $J_h$,
		we obtain
		\[
		J_h(s)
		=
		\int_s^1
		\bigl(G_h(\arccos x)+2\mathcal C_h\bigr)\,\mathrm dx,
		\qquad s\in[0,1].
		\]
		Therefore, for $t\in[0,1)$, Fubini's theorem gives
		\begin{align}\label{eqn:Nh-first}
			\begin{aligned}
				-\int_t^1J_h(s)\,\mathrm ds
				&=
				-\int_t^1
				\int_t^x
				\bigl(G_h(\arccos x)+2\mathcal C_h\bigr)
				\,\mathrm ds\,\mathrm dx\\
				&=
				-\int_t^1(x-t)G_h(\arccos x)\,\mathrm dx
				-\mathcal C_h(1-t)^2.
			\end{aligned}
		\end{align}
		Using \eqref{eqn:weighted-Gh} with $a=t$ and then
		\eqref{eqn:halfspace-Ch}, we obtain
		\begin{align}\label{eqn:Nh-numerator}
			\begin{aligned}
				-\int_t^1J_h(s)\,\mathrm ds
				=(1-t)^2
				\int_0^t\frac{y}{1-y^2}
				h(\arccos y)^2\,\mathrm dy
				+t
				\int_t^1
				\frac{1-y}{1+y}
				h(\arccos y)^2\,\mathrm dy.
			\end{aligned}
		\end{align}
		Here we split the integral over $[0,t]$
		and $[t,1]$, and used the identity
		\[
		(1-t)^2\frac{y}{1-y^2}
		-\frac{(y-t)(1-ty)}{1-y^2}
		=t\frac{1-y}{1+y}.
		\]
		By the definition of $Q[h]$, dividing
		\eqref{eqn:Nh-numerator} by $(1-t^2)^2$ yields precisely
		\eqref{eqn:halfspace-positive-swirl}.  Both terms in that formula are
		nonnegative for $0\leq t<1$.
		
		It remains to consider the endpoint $t=1$. By the endpoint calculation in Section~\ref{sec:integrated-reconstruction}, since $J_h(1)=0$ and
		\[
		\frac{\rmd J_h}{\rmd t} (1)
		=
		-\bigl(G_h(0)+2\mathcal C_h\bigr)
		=
		-2\mathcal C_h,
		\]
		we obtain
		\[
		\lim_{t\to1^-}Q[h](t)
		=
		-\frac{\mathcal C_h}{4}.
		\]
		Thus, defining $Q[h](1)=-\mcC_h/4$, we obtain $Q[h]\in C^0[0,1]$ and, by \eqref{eqn:halfspace-Ch}, $Q[h](1)\ge0$.
	\end{proof}
	
	The following lemma shows that $Q[h]$ depends quadratically on $h$. This estimate is used to prove the continuity of the fixed-point map and the uniqueness of solutions for small swirl forcing.
	
	\begin{lemma}[Quadratic estimates for the swirl contribution]
		\label{lm:Q-quadratic-estimate}
		Let $h_1,h_2\in C^1[0,\pi/2]$ satisfy $h_i(0)=h_i(\pi/2)=0$, and let $Q[h]$ be as in \eqref{eqn:halfspace-positive-swirl}. Then
		\begin{align*}
			\|Q[h_1]-Q[h_2]\|_{C^0[0,1]}
			\leq C
			\bigl(\|h_1\|_{C^1}+\|h_2\|_{C^1}\bigr)
			\|h_1-h_2\|_{C^1}.
		\end{align*}
		In particular,
		\[
		\|Q[h]\|_{C^0[0,1]}
		\leq C\|h\|_{C^1[0,\pi/2]}^2.
		\]
	\end{lemma}
	
	\begin{proof}
		Since $h_i(0)=0$ and $\vp/\sin\vp$ is bounded on $[0,\pi/2]$, we have $|h_i(\varphi)|\le C\sin\varphi\|h_i\|_{C^1}$. Set
		\[
		\mathcal B
		=\bigl(\|h_1\|_{C^1}+\|h_2\|_{C^1}\bigr)
		\|h_1-h_2\|_{C^1}.
		\]
		For $y=\cos\varphi$, a direct calculation shows
		\[
		\frac{|h_1(\arccos y)^2-h_2(\arccos y)^2|}{1-y^2}
		\leq C\mathcal B,
		\]
		and
		\[
		\frac{1-y}{1+y}
		|h_1(\arccos y)^2-h_2(\arccos y)^2|
		\leq C\mathcal B(1-y)^2.
		\]
		Subtracting the two representations in \eqref{eqn:halfspace-positive-swirl} and using these two estimates yields the first inequality. The last assertion follows by taking $h_2=0$.
	\end{proof}

	\subsection{Estimates for the linear swirl equation}\label{sec:ExistSwirl}
	
	In this subsection we establish two estimates for the linear swirl equation
	\eqref{eqn:swirlh}. The first is a basic $C^1$ estimate under an upper bound for $g$. 
	
	\begin{lemma}[Basic estimates for the linear swirl equation]
		\label{lm:ODEEst}
		Consider the boundary value problem
		\[
		(h' + h\cot \varphi)' = g\,(h' + h\cot \varphi) - Z,
		\quad  0< \vp <\pi/2,\qquad h(0)=h(\pi/2)=0,
		\]
		with $Z\in C^0[0,\pi/2]$.
		\begin{enumerate}[(i)]
			\item If $g\in C^0[0,\pi/2]$, there exists a unique solution
			$h\in C^1[0,\pi/2]$. It also satisfies
			$H=h'+h\cot\varphi\in C^1[0,\pi/2]$, with $H(0)=2h'(0)$.
			For every $B>0$ there is a constant $C_B$ such that
			\[
			\|h\|_{C^1[0,\pi/2]} \le C_B\|Z\|_{C^0[0,\pi/2]}
			\qquad\text{if }\|g\|_{C^0}\leq B.
			\]
			
			\item If $g_1,g_2$ have $C^0$ norm at most $B$ and $h_1,h_2$
			are the corresponding solutions for the same $Z$, then
			\[
			\|h_1-h_2\|_{C^1[0,\pi/2]}
			\le C_B\|Z\|_{C^0[0,\pi/2]}
			\|g_1-g_2\|_{C^0[0,\pi/2]}.
			\]
			
		\end{enumerate}
	\end{lemma}
	
	\begin{proof}
		(i) For simplicity we omit the interval $[0,\pi/2]$ in the norm notation.
		Set $H = h' + h\cot\varphi$.  From $H\sin\varphi = (h\sin\varphi)'$ and the boundary conditions,
		\[
		H' - gH = -Z,\qquad \int_0^{\pi/2} H(\varphi)\sin\varphi\,\mathrm{d}\varphi = 0.
		\]
		The linear ODE gives $H(\varphi) = \tilde g(\varphi)(H(0) - \tilde Z(\varphi))$ where
		$\tilde g = \exp\int_0^\varphi g$, $\tilde Z = \int_0^\varphi Z\exp(-\int_0^\psi g)\,\rmd \psi$.
		The integral condition yields
		\[
		H(0) = \frac{\int_0^{\pi/2}\tilde g \tilde Z \sin\varphi\,\mathrm{d}\varphi}
		{\int_0^{\pi/2}\tilde g \sin\varphi\,\mathrm{d}\varphi},
		\]
		hence $|H(0)| \le \|\tilde Z\|_{C^0}
		\le (\pi/2)e^{(\pi/2)\|g\|_{C^0}}\|Z\|_{C^0}$.
		Then
		\begin{align}\label{eqn:lm2EstH}
			\|H\|_{C^0} \le e^{(\pi/2)\|g\|_{C^0}}\bigl(|H(0)|+\|\tilde Z\|_{C^0}\bigr)
			\le \pi e^{\pi\|g\|_{C^0}}\|Z\|_{C^0}.
		\end{align}
		From $H' = gH - Z$ we get
		\[
		\|H'\|_{C^0} \le \|g\|_{C^0}\|H\|_{C^0} + \|Z\|_{C^0}
		\le \bigl(\pi e^{\pi\|g\|_{C^0}}\|g\|_{C^0} + 1\bigr)\|Z\|_{C^0},
		\]
		so that
		\[
		\|H\|_{C^1} \le C_B\|Z\|_{C^0}
		\qquad\text{when }\|g\|_{C^0}\le B.
		\]
		The function $h$ is recovered by
		\begin{align}\label{eqn:BVPKernel}
			h(\varphi) = \frac{1}{\sin\varphi}\int_0^\varphi H(\tau)\sin\tau\,\rmd\tau
			= \int_0^1 H(\varphi s)\frac{\varphi\sin(\varphi s)}{\sin\varphi}\,\rmd s.
		\end{align}
		The representation formula shows that $h(0)=0,\, h'(0)=\frac12H(0)$, while the integral condition gives $h(\pi/2)=0$. 
		The kernel
		$K(\varphi,s)=\varphi\sin(\varphi s)/\sin\varphi$ is smooth on
		$[0,\pi/2]\times[0,1]$. Hence $\|h\|_{C^{1}}\le C\|H\|_{C^{1}}$,
		which proves (i).
		
		\noindent(ii) Let $h_1,h_2$ correspond to $g_1,g_2$.  Set $\tilde h = h_1-h_2$, $\tilde H = \tilde h'+\tilde h\cot\varphi$, and $H_2 = h_2'+h_2\cot\varphi$.  Subtracting the equations,
		\[
		\tilde H' - g_1\tilde H = (g_1-g_2)H_2,\qquad \tilde h(0)=\tilde h(\pi/2)=0.
		\]
		Since $\tilde H$ fulfills the hypotheses of part (i) with $g=g_1$ and $Z$ replaced by $-(g_1-g_2)H_2$, its estimate gives
		\[
		\|\tilde h\|_{C^1}
		\le C_B\|H_2\|_{C^0}\,\|g_1-g_2\|_{C^0}.
		\]
		Using \eqref{eqn:lm2EstH} for $h_{2}$ we have
		\[
		\|H_2\|_{C^0} \le C_B\|Z\|_{C^0}.
		\]
		Combining the two inequalities gives the desired estimate.
	\end{proof}

	The preceding $C^1$ estimate depends exponentially on an upper bound for $g$, which is unavailable in the large-swirl fixed-point argument. We therefore derive a global estimate using only the lower bound $M\ge M_*$. It is based on a barrier and the weak maximum principle, and yields the bounds on $h$ and $Q[h]$ needed to construct the invariant set for arbitrarily large $\mu$.

	\begin{lemma}[Uniform estimate from a lower Riccati bound]
		\label{lm:uniform-swirl-estimate}
		Let $M,M_* \in C^0[0,1]$ satisfy $M \ge M_*$.
		Let $h$ be the solution supplied by Lemma~\ref{lm:ODEEst}
		with $Z$ replaced by $\mu Z$, that is,
		\[
		(h'+h\cot\varphi)'=g(h'+h\cot\varphi)-\mu Z,
		\qquad h(0)=h(\pi/2)=0,
		\]
		where $Z\in C[0,\pi/2]$, $\mu\in\mathbb R$, and $g(\varphi)=M(\cos\varphi)\sin\varphi$.
		There exists a constant $C(M_*)$, depending only on $M_*$ and independent of any upper bound for $M$, such that
		\[
		\|h\|_{C^0[0,\pi/2]}\le C(M_*)|\mu|\|Z\|_{C^0[0,\pi/2]}, \qquad 
		\|Q[h]\|_{C^0[0,1]}\le C(M_*)\mu^2\|Z\|_{C^0[0,\pi/2]}^2,
		\]
		where $Q[h]$ is defined in \eqref{eqn:halfspace-positive-swirl}.
	\end{lemma}
	
	\begin{proof}
		\underline{Step 1.} \emph{Change of variables.}
		Set $t=\cos\varphi$ and define
		\[
		V(t)=h(\arccos t)\sqrt{1-t^2}=h(\varphi)\sin\varphi.
		\]
		One can verify that 
		\[
		\frac{\rmd V}{\rmd t}(t)
		= -\frac{(h \sin \vp)'}{\sin\vp}
		=-h'-h\cot\varphi.
		\]
		Then the swirl equation becomes
		\[
		\sin \vp\, \frac{\rmd^2 V}{\rmd t^2}(t)
		= - M(\cos \vp) \sin \vp \,\frac{\rmd V}{\rmd t}(t) - \mu Z(\varphi),
		\]
		which gives
		\[
		-\frac{\rmd^2 V}{\rmd t^2}-M\frac{\rmd V}{\rmd t}
		=\frac{\mu Z(\arccos t)}{\sqrt{1-t^2}}, \quad\text{a.e. on }(0,1),
		\qquad V(0)=V(1)=0.
		\]
		One can verify that $ V\in C^1[0,1]\cap C^2[0,1)\cap W^{2,1}(0,1)$.
		This change of variables removes the singularity of $\cot\varphi$ at $\varphi=0$ from the drift coefficient and places it in the integrable term $(1-t^2)^{-1/2}$ on the right-hand side.
		
		\underline{Step 2.} \emph{Construction of the barrier.}
		We only know the lower bound $M\ge M_*$ and have no upper bound for $M$. Hence we cannot estimate $V$ directly by an upper bound for $M$. Instead, we construct a barrier $V_*$ depending only on $M_*$.
		
		Let $m_*=\min\limits_{[0,1]}M_*$. Define $u_*$ by
		\[
		\frac{\rmd u_*}{\rmd t}+m_*u_*=\frac1{\sqrt{1-t^2}},\qquad u_*(0)=0.
		\]
		Its explicit solution is
		\[
		u_*(t)=e^{-m_*t}\int_0^t\frac{e^{m_*s}}{\sqrt{1-s^2}}\,\rmd s.
		\]
		Since $(1-s^2)^{-1/2}\in L^1(0,1)$, we have $u_*\in C[0,1]\cap W^{1,1}(0,1)$ and $u_*\ge0$.
		Set $\delta=|\mu|\|Z\|_{C^0[0,\pi/2]}$ and define
		\[
		V_*(t)=\delta\int_t^1 u_*(s)\,\rmd s.
		\]
		Therefore $V_*\in C^1[0,1]\cap W^{2,1}(0,1)$ and $V_*\ge0$. Moreover, by the definition of $u_*$,
		\[
		\begin{aligned}
			-\frac{\rmd^2 V_*}{\rmd t^2}-M\frac{\rmd V_*}{\rmd t}
			=\delta\left(\frac{\rmd u_*}{\rmd t}+Mu_*\right)
			=\delta\left(\frac1{\sqrt{1-t^2}}+(M-m_*)u_*\right)
			\ge \frac{\delta}{\sqrt{1-t^2}}.
		\end{aligned}
		\]
		Indeed, since $u_*\ge0$ and $M\ge m_*$, the extra term is nonnegative. Thus the barrier inequality holds without any upper bound for $M$.

		\underline{Step 3.} \emph{Weak maximum principle and comparison.}
		We use the following standard weak maximum principle: if
		$w\in W^{2,1}(a,b)\cap C[a,b]$ satisfies
		\[
		-\frac{\rmd^2 w}{\rmd t^2}-M\frac{\rmd w}{\rmd t}\ge0
		\quad\text{a.e. on }(a,b),
		\qquad w(a)\ge0,\quad w(b)\ge0,
		\]
		then $w\ge0$ on $[a,b]$; see, e.g., \cite[Theorem~8.1]{GT01}.
		Now consider $V_*\pm V$. A direct computation gives
		\[
		-\frac{\rmd^2}{\rmd t^2}(V_*\pm V)-M\frac{\rmd}{\rmd t}(V_*\pm V)
		\ge \frac{\delta \pm \mu Z(\arccos t)}{\sqrt{1-t^2}}\ge0,
		\]
		where the last inequality follows from $|\mu Z|\le \delta$. Moreover,
		\[
		(V_*\pm V)(0)=V_*(0)\ge0,\qquad
		(V_*\pm V)(1)=0.
		\]
		Since $V_*\pm V\in W^{2,1}(0,1)\cap C[0,1]$, the weak maximum
		principle applies and yields $V_*\pm V\ge0$. 
		Hence $|V(t)|\le V_*(t)$.

		\underline{Step 4.} \emph{Recovery of the estimates.}
		Using $V(t)=h(\arccos t)\sqrt{1-t^2}$ and $|V(t)|\le V_*(t)$, we have
		\[
		|h(\arccos t)|\le \frac{V_*(t)}{\sqrt{1-t^2}}.
		\]
		Because $u_*$ is bounded by $C(M_*)$, we have  $V_*(t)=\delta\int_t^1 u_*(s)\,\rmd s\le C(M_*)\delta(1-t)$. Hence
		\[
		|h(\arccos t)|\le C(M_*)\delta\sqrt{\frac{1-t}{1+t}}
		\le C(M_*)\delta.
		\]
		Consequently, 
		\[
		\|h\|_{C^0[0,\pi/2]}\le C(M_*)|\mu|\|Z\|_{C^0[0,\pi/2]}.
		\]
		In particular, we also obtain $h(\varphi)=O(\varphi)$ as $\varphi\to0^+$.
		From the bound above,
		\[
		h(\arccos y)^2\le C(M_*)\delta^2\frac{1-y}{1+y}.
		\]
		Substituting this into the explicit formula
		\[
		\begin{aligned}
			Q[h](t)
			=
			\frac1{(1+t)^2}\int_0^t\frac{y}{1-y^2}h(\arccos y)^2\,\rmd y
			+ \frac{t}{(1-t^2)^2}\int_t^1\frac{1-y}{1+y}h(\arccos y)^2\,\rmd y,
		\end{aligned}
		\]
		we find that both integrals are bounded by $C(M_*)\delta^2$. Consequently,
		\[
		\|Q[h]\|_{C^0[0,1]}\le C(M_*)\mu^2\|Z\|_{C^0[0,\pi/2]}^2.
		\]
		This completes the proof.
	\end{proof}

	\subsection{Existence and uniqueness for the coupled equations}
	\label{sec:coupled-solvability}
	
	We now prove existence for the coupled system and uniqueness for small swirl forcing by a fixed-point argument.
	
	\begin{proposition}[Existence and uniqueness for the coupled equations]\label{prop:coupled-solvability}
		Suppose $N_*\in C^0[0,1]$ and the Riccati problem
		\[
		\frac{\rmd M_*}{\rmd t}+\frac12(M_*)^2=N_*,\qquad M_*(0)=0
		\]
		has a solution $M_*\in C^1[0,1]$.
		For every $\mu\in\mathbb R$ and $Z\in C[0,\pi/2]$, the coupled system
		\begin{equation}\label{eqn:abstract-coupled}
			\left\{
			\begin{aligned}
				\frac{\rmd M}{\rmd t}+\frac12M^2&=N_*+Q[h],\quad &&\text{in } (0,1),
				\quad M(0)=0,\\
				(h'+h\cot\varphi)'&=g(h'+h\cot\varphi)-\mu Z, \quad &&\text{in } (0,\pi/2),  \quad
				h(0)=h(\pi/2)=0,
			\end{aligned}
			\right.
		\end{equation}
		where $g(\varphi)=M(\cos\varphi)\sin\varphi$ and $Q[h]$ is defined in \eqref{eqn:halfspace-positive-swirl},
		has a solution with $M\in C^1[0,1]$, $h\in C^1[0,\pi/2]$,
		and $H=h'+h\cot\varphi\in C^1[0,\pi/2]$.
		Every such solution satisfies
		\begin{equation}\label{eqn:all-solutions-trapped}
			0\leq M-M_*\leq C(M_*)|\mu|^2\|Z\|_{C^0}^2,
			\qquad \|h\|_{C^0}\leq C(M_*)|\mu|\|Z\|_{C^0},
		\end{equation}
		where $C(M_*)$ depends only on $M_*$.
		
		Moreover, there exists $\varepsilon_*>0$, depending only on $M_*$, such that the solution is unique in the above regularity class whenever $|\mu|\|Z\|_{C^0}<\varepsilon_*$.
		For every constant $B>0$, the constants $C(M_*)$ and $\varepsilon_*$ may be
		chosen depending only on $B$ whenever $\|M_*\|_{C^0[0,1]}\leq B$.
	\end{proposition}

	\begin{proof}
		\underline{Step 1.} \emph{A priori estimates.}
		Let $(M,h)$ be any solution of \eqref{eqn:abstract-coupled}, and set $\delta=|\mu|\|Z\|_{C^0[0,\pi/2]}$. 
		Since $Q[h]\ge0$ by Lemma~\ref{lm:positive-swirl-contribution}, Lemma~\ref{lm:Riccati-comparison} gives $M\ge M_*$.
		Hence Lemma~\ref{lm:uniform-swirl-estimate} applies and yields
		\[
		0\le Q[h]\le C_1(M_*)\delta^2, \quad 
		\|h\|_{C^0}\le C_2(M_*)\delta.
		\]
		Substituting this into Lemma~\ref{lm:Riccati-comparison} gives
		\[
		0\le M-M_*\le C_3(M_*)\delta^2.
		\]
		Thus we get \eqref{eqn:all-solutions-trapped}, which holds for every solution.
		
		\underline{Step 2.} \emph{Construction of the invariant set.}
		By Step 1 every solution satisfies $M_*\le M\le M_*+C_3(M_*)\delta^2$. Set
		\[
		\mathfrak K=\bigl\{M\in C^0[0,1]:M(0)=0,\quad
		M_*\le M\le M_*+C_3(M_*)\delta^2\bigr\}.
		\]
		Then $\mathfrak K$ is a nonempty closed bounded convex subset of $C^0[0,1]$.
		
		Define the map $T$ on $\mathfrak K$ as follows. Given $M\in\mathfrak K$, set $g(\varphi)=M(\cos\varphi)\sin\varphi$, solve the linear swirl equation
		\[
		(h'+h\cot\varphi)'=g(h'+h\cot\varphi)-\mu Z,\qquad h(0)=h(\pi/2)=0
		\]
		for $h$, compute $Q[h]$, then solve the Riccati equation
		\[
		\frac{\rmd \widetilde M}{\rmd t}+\frac12\widetilde M^2=N_*+Q[h],\qquad \widetilde M(0)=0
		\]
		for $\widetilde M$, and set $T(M)=\widetilde M$. By Lemma~\ref{lm:ODEEst}, the linear swirl problem has a unique solution $h$; since $Q[h]\ge0$ and $Q[h]\in C[0,1]$, Lemma~\ref{lm:Riccati-comparison} yields a unique Riccati solution $\widetilde M$ on $[0,1]$. Thus $T$ is well defined. Equivalently, the construction follows the chain
		\[
		M
		\longmapsto
		g
		\longmapsto
		h
		\longmapsto
		Q[h]
		\longmapsto
		\widetilde M.
		\]
		Arguing as in Step 1, we obtain $0\le\widetilde M-M_*\le C_3(M_*)\delta^2$, hence $T(\mathfrak K)\subseteq\mathfrak K$.

		\underline{Step 3.} \emph{Schauder fixed point.}
		We use the Schauder fixed-point theorem to prove that $T$ has a fixed point, which yields existence for the coupled system. Step~2 already provides a closed convex set $\mathfrak K$ with $T(\mathfrak K)\subseteq\mathfrak K$. It remains to verify that $T$ is continuous and compact.
		
		We first prove that $T$ is continuous. Let $M_1,M_2\in\mathfrak K$, and denote the corresponding functions by $g_i,h_i$, and $\widetilde M_i=T(M_i)$, $i=1,2$. Since $g_i=M_i(\cos\varphi)\sin\varphi$, smallness of $M_1-M_2$ implies smallness of $g_1-g_2$. By Lemma~\ref{lm:ODEEst}, applied with $Z$ replaced by $\mu Z$,
		\[
		\|h_i\|_{C^1}\le C_{\mathfrak K}\delta,
		\qquad
		\|h_1-h_2\|_{C^1}\le C_{\mathfrak K}\delta\|M_1-M_2\|_{C^0},
		\]
		where $C_{\mathfrak K}$ depends on $\|M_*\|_{C^0[0,1]}+C_3(M_*)\delta^2$. Lemma~\ref{lm:Q-quadratic-estimate} then gives
		\[
		\|Q[h_1]-Q[h_2]\|_{C^0}\le C_{\mathfrak K}\delta^2\|M_1-M_2\|_{C^0}.
		\]
		Since $\widetilde M_i\in\mathfrak K$, we have $\|\widetilde M_i\|_{C^0}\le \|M_*\|_{C^0[0,1]}+C_3(M_*)\delta^2$. Hence the integrating-factor estimate for the Riccati difference equation yields
		\begin{align}\label{prop:eqn-difference-est}
			\|T(M_1)-T(M_2)\|_{C^0}\le C_{\mathfrak K}\delta^2\|M_1-M_2\|_{C^0}.
		\end{align}
		Thus $T$ is continuous.
		We next prove that $T$ is compact. For every $M\in\mathfrak K$, the function $\widetilde M=T(M)$ satisfies
		\[
		\frac{\rmd\widetilde M}{\rmd t}=N_*+Q[h]-\frac12\widetilde M^2.
		\]
		Since $Q[h]$ and $\widetilde M$ are uniformly bounded on $\mathfrak K$, the derivative $\frac{\rmd\widetilde M}{\rmd t}$ is also uniformly bounded. Hence $T(\mathfrak K)$ is bounded in $C^1[0,1]$ and, by the Arzel\`a--Ascoli theorem, relatively compact in $C^0[0,1]$. Thus $T$ is compact.
		
		The Schauder fixed-point theorem (Lemma~\ref{lm:Schauder}) therefore gives $M\in\mathfrak K$ with $T(M)=M$. By the definition of $T$, this $M$ and the corresponding $h$ solve \eqref{eqn:abstract-coupled}, which proves existence. The asserted regularity follows from the Riccati equation with continuous right-hand side and from Lemma~\ref{lm:ODEEst}.
		
		\underline{Step 4.} \emph{Uniqueness for small swirl.}
		Assume $\delta\le1$. Then the coefficient in \eqref{prop:eqn-difference-est} can be taken as $C(M_*)$, so
		\[
		\|T(M_1)-T(M_2)\|_{C^0}\le C(M_*)\delta^2\|M_1-M_2\|_{C^0}.
		\]
		Choose $0<\varepsilon_*\le1$ such that $C(M_*)\varepsilon_*^2<1$. If $\delta=|\mu|\|Z\|_{C^0}<\varepsilon_*$, then $T$ is a contraction on $\mathfrak K$ and hence has a unique fixed point. 
		Every solution of \eqref{eqn:abstract-coupled} gives a fixed point of $T$, and Step~1 shows that $M \in \mathfrak K$. Therefore, uniqueness of the fixed point in
		$\mathfrak K$ implies uniqueness in the entire regularity class. Once $M$ is unique, $g$ is unique, and the linear swirl problem determines $h$ uniquely, so $(M,h)$ is unique.

		Finally, the constants defining $C(M_*)$ in
		\eqref{eqn:all-solutions-trapped} and the uniqueness threshold
		$\varepsilon_*$ depend on $M_*$ only through
		$\|M_*\|_{C^0[0,1]}$. Hence, if
		\[
		\|M_*\|_{C^0[0,1]}\leq B
		\]
		for some constant $B>0$,
		then $C(M_*)$ and $\varepsilon_*$ can be chosen depending only on
		$B$.
	\end{proof}

	\subsection{Proofs of the main results}
	
	\begin{proof}[\textbf{Proof of Theorem~\ref{thm:main}(ii)}]
		Fix $\lambda\in I$, and let $M_\lambda$ be the corresponding
		no-swirl Riccati solution. By \eqref{eqn:N-decomposition},
		Proposition~\ref{prop:coupled-solvability}, applied with
		$N_*=\lambda N_0$ and $M_*=M_\lambda$, gives a solution of the
		reduced system for every $\mu\in\mathbb R$, as well as uniqueness
		under the stated smallness condition. Lemma~\ref{lm:reconstruction}
		then gives a bounded-profile solution of \eqref{eqn:NSEhfD} with the stated regularity. 
		
		Finally, $M_\lambda$ depends continuously on $\lambda\in I$ and is
		therefore uniformly bounded when $\lambda$ ranges over
		$K\Subset I$. The uniformity statement in
		Proposition~\ref{prop:coupled-solvability} then gives a common
		uniqueness threshold for all $\lambda\in K$.
	\end{proof}
	
	\begin{proof}[\textbf{Proof of Corollary~\ref{cor:pure-swirl}}]
		By \cite[Theorem~5.1]{KMT18}, the unforced problem has only the zero
		solution. Applying Theorem~\ref{thm:main}(ii) with $\lambda=0$ gives
		existence for every $\mu\in\mathbb R$.
		
		Let $(M,h)$ be any such solution with $\mu Z\not\equiv0$. The swirl
		equation gives $h\not\equiv0$, and hence
		\[
		\frac{\rmd M}{\rmd t}
		+
		\frac12M^2
		=
		Q[h],
		\qquad
		M(0)=0,
		\]
		where $Q[h](t)>0$ for $0<t\leq1$ by
		Lemma~\ref{lm:positive-swirl-contribution}. The integrating-factor
		formula gives
		\[
		M(t)
		=
		\int_0^t
		\exp\left(
		-\frac12\int_s^tM(r)\,\mathrm dr
		\right)
		Q[h](s)\,\mathrm ds
		>0
		\qquad\text{for }0<t\leq1.
		\]
		Therefore, $g(\varphi) = M(\cos\varphi)\sin\varphi >0$
		for $0<\varphi<\frac{\pi}{2}$,
		which proves the assertion.
	\end{proof}

	\begin{proof}[\textbf{Proof of Corollary~\ref{cor:small-tangential-curl}}]
		\underline{Step 1.} \emph{Existence and uniqueness.}
		Write $\BF = \frac1{\rho^3}
		\bigl( X(\varphi)\Be_\rho + Y(\varphi)\Be_\varphi + Z(\varphi)\Be_\theta \bigr)$.
		By \eqref{eqn:sphericalcurl},
		\[
		\|X'+2Y\|_{C^0[0,\pi/2]}
		\leq
		\eta,
		\qquad
		\|Z\|_{C^0[0,\pi/2]}
		\leq
		\frac{\eta}{2}.
		\]
		If $\eta<2/\pi$, Proposition~\ref{prop:small-poloidal} gives a
		no-swirl solution $M_*$ satisfying
		\[
		\|M_*\|_{C^1[0,1]}
		\leq
		\pi\eta.
		\]
		Proposition~\ref{prop:coupled-solvability} with $\mu=1$ then gives a
		solution of the coupled system, which is reconstructed as a
		bounded-profile solution by Lemma~\ref{lm:reconstruction}.
		Since $\|M_*\|_{C^0}\leq\pi\eta$ and $\|Z\|_{C^0}\leq\frac{\eta}{2}$, the uniformity statement in Proposition~\ref{prop:coupled-solvability} gives an absolute constant
		$\varepsilon>0$ such that the solution is unique whenever
		$\eta\leq\varepsilon$.
		
		\underline{Step 2.} \emph{Profile estimates.}
		Step~1 implies $\|M_*\|_{C^1}\leq C\eta$. Proposition~\ref{prop:coupled-solvability}, together with the
		difference equation for $M-M_*$, gives
		\[
		\|M-M_*\|_{C^1}\leq C\|Z\|_{C^0}^2\leq C\eta^2.
		\]
		Hence
		\[
		\|M\|_{C^1}
		\leq
		\|M_*\|_{C^1}
		+
		\|M-M_*\|_{C^1}
		\leq C\eta.
		\]
		Lemma~\ref{lm:ODEEst} also gives
		\[
		\|h\|_{C^1}\leq C\|Z\|_{C^0}\leq C\eta.
		\]
		Finally, using $g(\varphi)=M(\cos\varphi)\sin\varphi$ and $f=-g'-g\cot\varphi$,
		we obtain \eqref{eqn:small-curl-velocity-bound}.
	\end{proof}

	\subsection{An example with \texorpdfstring{$\lambda\notin I$}{lambda outside I}}
	\label{sec:outside-I}
	
	Theorem~\ref{thm:main}(i) shows that the no-swirl problem has a bounded-profile solution if and only if $\lambda\in I$. In contrast, the following example shows that even for some $\lambda\notin I$, a suitably chosen nonzero swirl force can still yield a bounded-profile solution.
	
	\begin{example}[Solutions with $\lambda\notin I$]\label{ex:outside-no-swirl-interval}
		Define the smooth axisymmetric $(-3)$-homogeneous fields on
		$\mathbb R^3\setminus\{0\}$ by
		\[
		\BF_0(\Bx)
		=
		-\frac{x_3^2}{|\Bx|^6}(x_1,x_2,0),
		\qquad
		\BF_1(\Bx)
		=
		\frac{6x_3}{|\Bx|^5}(-x_2,x_1,0),
		\]
		where $\BF_0$ is swirl-free and $\BF_1$ is purely swirling.
		For \eqref{eqn:NSEhfD} with force
		$\BF_{\lambda,\mu}=\lambda\BF_0+\mu\BF_1$, the no-swirl
		solvability interval is
		\[
		I=(-\infty,\overline\lambda),
		\qquad
		0<\overline\lambda
		\leq
		400 \pi^2
		<4000.
		\]
		When swirl is allowed, for every $a\in\mathbb R$, the force $\BF_{a^2,a}=a^2\BF_0+a\BF_1$ admits the explicit bounded-profile solution
		\begin{equation}\label{eq:outside-I-explicit-velocity}
			\Bu_a(\Bx)
			=
			\frac{ax_3}{|\Bx|^3}(-x_2,x_1,0),
			\qquad
			p_a=0.
		\end{equation}
		In particular, $(\lambda,\mu)=(4000,20\sqrt{10})$ admits a bounded-profile
		solution although $4000\notin I$.
	\end{example}

	\begin{proof}
		\underline{Step 1.} \emph{Verification of the explicit solution.}
		Direct differentiation gives
		\[
		\operatorname{div}\Bu_1=0,\qquad
		-\Delta\Bu_1=\BF_1,\qquad
		(\Bu_1\cdot\nabla)\Bu_1=\BF_0.
		\]
		It follows that
		\[
		-\Delta\Bu_a+(\Bu_a\cdot\nabla)\Bu_a
		=a\BF_1+a^2\BF_0.
		\]
		Moreover, $\Bu_a$ is $(-1)$-homogeneous and smooth away from the
		origin. Its only nonzero angular profile is $h_a(\varphi)=a\sin\varphi\cos\varphi$,
		which is bounded and vanishes at $\varphi=0$ and $\varphi=\pi/2$.
		Hence $\Bu_a$ is a bounded-profile solution in the sense of
		Definition~\ref{def:bounded-profile}.
		
		\underline{Step 2.} \emph{The no-swirl problem.}
		In spherical coordinates,
		\[
		\BF_0
		=
		\frac1{\rho^3}
		\bigl(X(\varphi)\Be_\rho+Y(\varphi)\Be_\varphi\bigr),
		\]
		where $X(\varphi)=-\sin^2\varphi\cos^2\varphi$, and $Y(\varphi)=-\sin\varphi\cos^3\varphi$. 
		In the notation of Section~\ref{sec:reductNo-swirl}, the profiles
		$X,Y$ of $\BF_0$ yield
		\[
		G_0(\arccos t)=\frac32t^4-t^2-\frac12,\qquad
		\mathcal C_0=\frac14,\qquad
		J_0(t)=-\frac1{30}+\frac{t^3}{3}-\frac{3t^5}{10}.
		\]
		Consequently, the Riccati forcing is
		\begin{equation}\label{eqn:example-formula-N0}
			N_0(t)=-\frac{t(3t^3+6t^2+4t+2)}{60(1+t)^2}. 
		\end{equation}
		Since $N_0(t)<0$ for $t\in(0,1]$, Theorem~\ref{thm:main}(i) yields $I=(-\infty,\overline\lambda)$ for some finite  $\overline\lambda>0$.
		
		\underline{Step 3.} \emph{Upper bound for $\overline\lambda$.}
		The function $-N_0$ is increasing on $(0,1]$ and satisfies $N_0\left(\frac12\right) = -\frac{47}{2160} < -\frac1{50}$.
		Hence
		\[
		N_0(t)\leq-\frac1{50}
		\qquad\text{for }\frac12<t<1.
		\]
		Applying Remark~\ref{rmk:rough-Riccati-bound} with $C=\frac{1}{50},\, \delta=\frac12,$
		we obtain
		\[
		\overline\lambda
		\leq 400\pi^2 < 4000.
		\]
		Thus $4000\notin I$. On the other hand, taking $a=20\sqrt{10}$ in Step~1 gives a bounded-profile solution for $\BF_{4000,20\sqrt{10}}$. Therefore, $(\lambda,\mu)=(4000,0)$ admits no bounded-profile solution, while $(\lambda,\mu)=(4000,20\sqrt{10})$ does.
		This completes the construction of the example.
	\end{proof}

	\medskip
	
	{\bf Acknowledgments.} The research of Wang is partially supported by NSFC grants 12671279, the Natural Science Foundation of Jiangsu Province (Grant No. BK20240147). The research of Xie is partially supported by  NSFC grants 12571238 and 12426203.

	\bibliographystyle{plain}
	\bibliography{ref}
	
\end{document}